\documentclass[a4paper,12pt]{article}
\usepackage{amsfonts,amsthm,amsmath,latexsym,bm,graphics,indentfirst}
\usepackage[hypertex]{hyperref}

\newtheorem{theorem}{Theorem}[section]
\newtheorem{lemma}[theorem]{Lemma}
\newtheorem{proposition}[theorem]{Proposition}
\newtheorem{corollary}[theorem]{Corollary}
\newtheorem{remark}[theorem]{Remark}

\newcommand{\BE}{\begin{equation}}
\newcommand{\BEN}{\begin{equation*}}
\newcommand{\EE}{\end{equation}}
\newcommand{\EEN}{\end{equation*}}
\newcommand{\BL}{\begin{lemma}}
\newcommand{\EL}{\end{lemma}}
\newcommand{\BT}{\begin{theorem}}
\newcommand{\ET}{\end{theorem}}
\newcommand{\BP}{\begin{proposition}}
\newcommand{\EP}{\end{proposition}}
\newcommand{\BC}{\begin{corollary}}
\newcommand{\EC}{\end{corollary}}
\newcommand{\BR}{\begin{remark}}
\newcommand{\ER}{\end{remark}}
\newcommand{\D}{\displaystyle}
\newcommand{\F}[2]{\frac{#1}{#2}}
\newcommand{\e}{\varepsilon}
\allowdisplaybreaks

\title{Infinitely many solutions for the prescribed boundary mean curvature problem in $\mathbb B^N$}
\author{Liping Wang\footnotemark[2] \qquad Chunyi Zhao\footnotemark[1]}

\begin{document}

\footnotetext[1]{Corresponding author. \\ \indent \ \ \
                 Department of Mathematics, East China Normal University, Shanghai, 200241, China. Email: cyzhao@math.ecnu.edu.cn.}
\footnotetext[2]{Department of Mathematics, East China Normal University, Shanghai, 200241, China. Email: lpwang@math.ecnu.edu.cn.}

\maketitle

\begin{abstract}
We consider the following prescribed boundary mean curvature problem in $ \mathbb B^N$ with the Euclidean metric
\BEN
\begin{cases}
\D -\Delta u  =0,\quad  u>0  &\text{in }\mathbb B^N, \\[2ex]
\D \F{\partial u}{\partial\nu} + \F{N-2}{2} u =\F{N-2}{2} \widetilde K(x) u^{2^\#-1} \quad & \text{on }\mathbb S^{N-1},
\end{cases}
\EEN where $\widetilde K(x)$ is positive and rotationally symmetric on $\mathbb
S^{N-1}, 2^\#=\F{2(N-1)}{N-2}$.
 We show that if $\widetilde K(x)$ has a local maximum point,  then the above problem has {\bf infinitely many positive} solutions,
 which are not rotationally symmetric on $\mathbb S^{N-1}$.
\end{abstract}


\section{Introduction}\label{s1}

Parallel to the prescribed scalar curvature problem, the prescribed
boundary mean curvature problem also plays an important role in
conformal geometry. Given an $N$-dimensional ($N\geq 3$) Riemannian
manifold $(M, g)$ with boundary, this problem concerns if one can
find a new metric $\tilde g$ in the conformal class of $g$, such
that $(M,\tilde g)$ has zero scalar curvature and the boundary mean
curvature becomes a prescribed function. Denote $\tilde g =
u^\F{4}{N-2} g$ where $u$ is a positive smooth function, then
problem may be addressed to finding a positive solution $u$ of the
coming equation \BEN
\begin{cases}
\D -\F{4(N-1)}{N-2}\Delta_g u + R_g u =0  &\text{in }M, \\[2ex]
\D \F{\partial u}{\partial\nu} + \F{N-2}{2}H_g u =\F{N-2}{2}\widetilde  K(x) u^{2^\#-1} \quad & \text{on }\partial M,
\end{cases}
\EEN
where $2^\#=\F{2(N-1)}{N-2}$ is the critical exponent of the Sobolev trace embedding. Here $\Delta_g$ is the Laplace-Beltrami operator,
$R_g$ is the scalar curvature of $M$,
$H_g$ is the mean curvature of $\partial M$, $\nu$ is the outward normal unit vector with respect to the
metric $g$ and $\widetilde K(x)$ is the prescribed function.\par

Due to the fact that the embedding $H^1(M)\hookrightarrow
L^{2^\#}(\partial M)$ is not compact, the Euler-Lagrange functional
$J$ associated to our problem fails to satisfy the Palais-Smale
condition. That is there exists noncompact sequence along which the
functional $J$ is bounded and its gradient goes to zero. Therefore,
it is not possible to apply the standard variational methods to
prove the existence of solutions. Notice that the above problem is a
natural analogue to the well-known scalar curvature problems on
closed manifolds.

Escobar \cite{E1,E3} and Marques \cite{M1,M2} studied this problem
for the case $\widetilde K(x)$ is a constant. From the existence of
solutions, they showed in this case that most compact manifolds with
boundary are conformally diffeomorphic to a manifold that resembles
the ball in two ways, namely, it has zero scalar curvature and its
boundary has constant mean curvature, although very few regions are
really conformal to the ball in higher dimensions. About other
related results we refer to the works \cite{A,C,E2} and the
references therein. \par


In this paper, we prescribe mean curvature on the boundary $\mathbb S^{N-1}$ of the unit ball $\mathbb B^N$ in
$\mathbb R^N$ ($N\geq 3$) with Euclidean metric $g_0$. Precisely we study the problem of finding a conformal
metric to $g_0$ whose scalar curvature vanishes in $\mathbb B^N$ and the mean curvature of boundary $\mathbb S^{N-1}$ is given by $\widetilde K(x)$.
This problem is equivalent to solving the following boundary problem
\BE\label{1.3}
\begin{cases}
\D -\Delta u  =0,\quad  u>0  &\text{in }\mathbb B^N, \\[2ex]
\D \F{\partial u}{\partial\nu} + \F{N-2}{2} u =\F{N-2}{2}\widetilde  K(x) u^{2^\#-1} \quad & \text{on }\mathbb S^{N-1},
\end{cases}
\EE Note that Cherrier \cite{CH} studied the regularity  for this
equation. He showed that solutions of (\ref{1.3}) which are of class
$H^1$ are also smooth.\par

The problem of determining which $\widetilde K(x)$ admits a solution to
(\ref{1.3}) has been studied extensively. It is easy to see that a
necessary condition for solving the problem is that $\widetilde K(x)$ has to be
positive somewhere. But there are also some obstructions for the
existence of  solutions, which are said of \textit{topological
type}. For example, the solution $u$ must satisfy the following
Kazdan-Warner condition (see \cite{E3}) \BE\label{1.2} \int_{\mathbb
S^{N-1}}\nabla\widetilde  K \cdot  x  u^{2^\#} \mathrm dx= 0. \EE Some
existence results have been obtained under some assumptions
involving the Laplacian at the critical points of $\widetilde K$. Sufficient
conditions in dimensions $3$ and $4$ are given in \cite{EG} and
\cite{DMA}. Furthermore in \cite{ACA}, the authors developed a Morse
theoretical approach to this problem in the $4$-dimensional case
providing some multiplicity results under generic conditions on the
function $\widetilde K$.

Consider the case $\widetilde K(x)=1+\e h(x)$ is a perturbation of $1$ (or
generally  a perturbation of some constant). In \cite{CXY}, by a
perturbation method, Chang, Xu and Yang obtained positive solutions
by looking for constrained minimizers, more precisely, they proved
that if at each critical point $Q$ of $h(x)$, $\Delta_{\mathbb
S^{N-1}}h(Q)= 0$, then under additional conditions, the above
problem has a positive solution for $\e$ sufficiently small.
Furthermore, Cao-Peng \cite{CP} constructed a two-peak solution
whose maximum points are located near two critical points of $h$ as
$\e\to 0$ under certain assumptions.\par

It is well known that the unit ball $\mathbb B^N$ is conformal to
the half-space $\mathbb R_+^N$. As in \cite{CP}, to consider this
problem we transfer the equation (\ref{1.3}) to an  equation in the
half-space $\mathbb R_+^N$. We denote
$y=(y_1,\cdots,y_N)=(y',y_N)\in \mathbb B^N $.  By the standard
stereographic projection: $\Pi$: $\mathbb B^N \to \mathbb R_+^N $,
\begin{align*}
\Pi(y',y_N) &= \left(\F{4y'}{(1+y_N)^2+|y'|^2},\F{2(1-y_N^2-|y'|^2)}{(1+y_N)^2+|y'|^2}\right),\\
\tilde u(x)&= \F{4^\F{N-2}{2}u(\Pi^{-1}
x)}{\left[(2+x_N)^2+|x'|^2\right]^\F{N-2}{2}},
\end{align*}
 we see that the function $\tilde u(x)$ satisfies
\BE\label{1.1}
\begin{cases}
\Delta u = 0,\ \ u>0   \qquad & \text{in }\mathbb R^N_+, \smallskip \\
\D \F{\partial u}{\partial \nu} = K(x) u^{2^\#-1}   & \text{on }\partial\mathbb R^N_+, \smallskip \\
u\in D^{1,2}(\mathbb R_+^N),
\end{cases}
\EE where $D^{1,2}(\mathbb R_+^N)$ denotes the completion of
$C_0^\infty(\overline{\mathbb R_+^N})$ under the norm $\int_{\mathbb
R_+^N} |\nabla u|^2$, the bounded function $K=\widetilde
K\circ\Pi^{-1}$.\par

For the case that $K(x)$ is a positive constant, say $1$ for
convenience, it is well-known from \cite{LZ}  that the only solution
to (\ref{1.1}) has the following form
$$
U_{\zeta,\Lambda}(x)=\left(N-2\right)^\frac{N-2}{2}\left[\frac{\Lambda
}{(1+\Lambda x_N)^2+\Lambda^2|\bar
x-\bar\zeta|^2}\right]^{\frac{N-2}{2}},
$$
where both $\Lambda>0$ and  $\bar\zeta\in \mathbb R^{N-1}$ are
arbitrary. Obviously it
 is  radially symmetric in $\partial\mathbb R_+^N$
with respect to $\bar \zeta$. Here we write $x=(\bar{x},x_N)$,
$\bar{x} \in \mathbb R^{N-1}$.

In this paper, we consider the simplest general case, i.e. $K(x)=K(|\bar
x|)=:K(r)$ is a radially symmetric \textit{positive} function in
$\partial\mathbb R_+^N$. The Kazdan-Warner condition (\ref{1.2}) is
correspondingly deduced to \BEN \int_{\mathbb R^{N-1}} K'(r) r
u^{2^\#} \mathrm d\bar x=\int_{\mathbb R^{N-1}} (\nabla K(\bar x)
\cdot \bar x)u^{2^\#} \mathrm d\bar x= 0. \EEN Hence by positiveness
of $u$, $K'(r)$ cannot have  fixed sign in $\mathbb R^{N-1}$. Thus
it is natural to assume that $K$ is \textit{not monotone}.\par

The purpose of this paper is to answer the following two questions:

\medskip

\noindent {\bf \it Q1: Does the existence of a local maximum of $K$
guarantee the existence of  solutions to (\ref{1.1})?}

\smallskip

\noindent
{\bf \it Q2: Are there solutions  to (\ref{1.1}) which are non-radially symmetric in $\partial\mathbb R_+^N$?}
\medskip

To state the main result, we assume that $K(r)$  satisfies the
following condition: \textit{$K(x)$ is positive, bounded and there
is a constant $r_0>0$, such that}
\[
K(r)= K(r_0)- c_0|r-r_0|^m+ O(|r-r_0|^{m+\theta})\quad \text{for }r\in
(r_0-\delta, r_0+\delta) \tag{K}  \label{K}
\]
where $c_0>0$, $\theta>0, \delta>0$ are some constants and the
constant $m$ satisfies  $m\in [2,N-2)$. To make sure that such $m$
exists, we consider the problem for $N \ge 5$. Without loss of
generality, we assume that
\[
K(r_0)=1.
\]

Our main result is stated as follows.

\BT\label{t1} Suppose that $N\ge 5$.  If $K(r)$ satisfies (\ref{K}),
then problem (\ref{1.1}) has infinitely many solutions, which is
non-radial in $\partial\mathbb R_+^N$. \ET

\BR Combining the results in  \cite{EG} and \cite{DMA}, we give
sufficient conditions for the existence of solutions for all $N \ge
3$.

\ER

\BR
The condition (\ref{K}) is a local condition while the condition in \cite{ACA} is global.
\ER

\begin{remark}
Theorem \ref{t1} exhibits a {\bf new phenomena} for the prescribed boundary mean curvature problem.
It suggests that if the critical points of $K$ are not isolated, new solutions to (\ref{1.1}) may bifurcate.
\end{remark}

We formulate the following conjecture in the general case.\par
\smallskip
\noindent {\it {\bf Conjecture:} If the set $ \{x\in\partial\mathbb
R_+^N: K(x)= \max_{x \in \partial\mathbb R_+^N} K(x) \}$ is an
 $\ell$-dimensional smooth manifold without
boundary, where $ 1\leq \ell  \leq N-2$. Then problem (\ref{1.1})
admits infinitely many positive solutions.}\par\medskip

Let us outline the main idea in the proof of Theorem \ref{t1}.
Let us fix a positive integer
\BEN
k\geq k_0
\EEN
where $k_0$ is a large integer, which is to be determined later. Set
\BEN
\mu=k^\F{N-2}{N-2-m}
\EEN
be the scaling parameter.


Using the transformation $u(y)\mapsto\mu^{-\F{N-2}{2}}u(\F{y}{\mu})$, we note that (\ref{1.1}) is equivalent to
\BE\label{8}
\begin{cases}
\Delta u = 0   & \text{in }\mathbb R^N_+, \smallskip \\
\D \F{\partial u}{\partial \nu} =  K(\F{|y|}{\mu}) u^{2^\#-1} \quad & \text{on }\partial\mathbb R^N_+.
\end{cases}
\EE

In the paper,  let \BEN
x_j=\left(r\cos\F{2(j-1)\pi}{k},r\sin\F{2(j-1)\pi}{k},0,\cdots,0\right),
\qquad j=1,\ldots,k, \EEN
then the approximation solution we choose
is \BEN W_{r,\Lambda}(y)= \sum_{j=1}^k U_{x_j,\Lambda}=\left(N-2\right)^\frac{N-2}{2}\sum_{j=1}^k
\left[\F{\Lambda}{(1+\Lambda y_N)^2+\Lambda^2|\bar y-\bar
x_j|^2}\right]^\F{N-2}{2}. \EEN We will find the solution with the
form $W_{r,\Lambda}+\phi$, furthermore $\phi$ has the following
symmetries
\begin{gather}
\phi(y_1,y_2,\cdots,y_{N-1},y_N)=\phi(y_1,-y_2,\cdots,-y_{N-1},y_N), \tag{i} \label{i}\\
\phi(y)=\phi(Q_ky), \qquad Q_k=
\left(
\begin{array}{c|c}
 \begin{array}{cc}
   \cos\F{2\pi}{k} & -\sin\F{2\pi}{k} \\[1ex]
   \sin\F{2\pi}{k} & \cos\F{2\pi}{k}
 \end{array}
  & \  {0} \\[2.5ex]
 \hline  \\[-1.5ex]
{ 0}  & \ I
\end{array}
\right),  \tag{ii}\label{ii}
\end{gather}
where $I$ denotes the $(N-2)\times(N-2)$ identical matrix. In the whole paper,
we always assume that
\begin{gather*}
r\in\left[\mu r_0-\F{1}{\mu^{\bar \theta}}, \mu r_0+\F{1}{\mu^{\bar
\theta}}\right], \qquad \qquad  L_0\leq \Lambda\leq L_1,
\end{gather*}
where $\bar \theta>0$ is a small number and $L_1> L_0>0$.

Theorem \ref{t1} is a direct consequence of the following theorem.

\BT\label{t2} Suppose $N\geq 5$. If $K$ satisfies (\ref{K}), then
there is an integer $k_0>0$ such that for any integer $k>k_0$,
problem (\ref{8}) has a solution $u_k$ of the form \BEN u_k =
W_{r_k,\Lambda_k} + \phi_k, \EEN where $\phi_k$ satisfies (\ref{i})
and (\ref{ii}). Moreover, as $k\to\infty$, $\|\phi_k\|_\infty \to
0$, $r_k\in\left[\mu r_0-\F{1}{\mu^{\bar \theta}}, \mu
r_0+\F{1}{\mu^{\bar \theta}}\right]$ and $L_0\leq \Lambda_k\leq
L_1$. \ET

\BR Changing back the solutions in Theorem \ref{t2}, we see that the
solutions to (\ref{1.3}) can blow up at arbitrarily large number of
points on $\mathbb S^{N-1}$. On the other hand, Escobar-Garcia
\cite{EG} shows that  when $N\geq 4$ and the function $K$ at its
critical points vanishes up to order $m$ with $m>N-2$, there is
actually at most one possible blow-up point. Thus our existence
result means that $m<N-2$ is almost sharp. \ER

We will use the finite reduction method introduced by Wei-Yan \cite{WY} to prove Theorem \ref{t2},
in which the authors use $k$, the number of the bubbles of
the solutions, as the parameter in the construction of bubbles solutions for (\ref{8}).
The main difficulty in constructing solution with $k$-bubbles is that we need to obtain a better
control of the error terms.
Since the maximum norm will not be affected by the number of the bubbles, we will carry out the reduction procedure in a space
with weighted maximum norm.\par

Our paper is organized as follows. In Section \ref{s2}, we get some
preliminary estimates. In Section \ref{s3}, we deal with the
corresponding linearized and nonlinear problems. In Section
\ref{s4}, we come to the variational reduction procedure. In Section
\ref{s5}, the proof of Theorem \ref{t2} is given. Finally we give
the energy expansion of the approximation solution and list some
useful estimates in the appendix Section \ref{s6}.\par

Throughout this paper, $C$ is a various generic constant independent of $k$ and $\mu$.\par
\medskip

\section{Preliminary Estimates} \label{s2}

In this section we will get some estimates for the posterior use.\par

Under the assumption that the solution $u=W_{r,\Lambda} + \phi$, it
is not difficult to check that $\phi$ should satisfy the following
equation \BE\label{5}
\begin{cases}
-\Delta \phi =0 & \text{in }\mathbb R^N_+ , \\
\D \F{\partial\phi}{\partial\nu} - (2^\#-1) K(\F{|y|}{\mu})W_{r,\Lambda}^{2^\#-2}\phi= - R(y) + N(\phi) \quad &\text{on }\partial\mathbb R^N_+ ,
\end{cases}
\EE
where the error term $R(y)$ and the nonlinear term $N(\phi)$ is defined by
\begin{align*}
R(y) &= \F{\partial W_{r,\Lambda}}{\partial\nu} - K(\F{|y|}{\mu})W_{r,\Lambda}^{2^\#-1} ,\\
N(\phi)&= K(\F{|y|}{\mu}) \left[(W_{r,\Lambda}+\phi)^{2^\#-1}-W_{r,\Lambda}^{2^\#-1} - (2^\#-1)W_{r,\Lambda}^{2^\#-2}\phi\right].
\end{align*}\par

In what follows, we use the following two important weighted norms
\begin{align*}
\|\phi\|_* &= \sup_{y\in\overline{\mathbb R_+^{N}}}\left(\sum_{j=1}^k \frac {1}{(1+| y- x_j|)^{\F{N}{2}-\F{m}{N-2}+\tau}}\right)^{-1}|\phi(y)|,\\
\|h\|_{**} &= \sup_{y\in\partial\mathbb R_+^{N}}\left(\sum_{j=1}^k
\frac {1}{(1+| y-
x_j|)^{\F{N+2}{2}-\F{m}{N-2}+\tau}}\right)^{-1}|h(y)|\\
&=\sup_{\bar y\in\mathbb R^{N-1}}\left(\sum_{j=1}^k \frac
{1}{(1+|\bar y-\bar
x_j|)^{\F{N+2}{2}-\F{m}{N-2}+\tau}}\right)^{-1}|h(y)|
\end{align*}
where $0<\tau<\frac{1}{2(N-2)}$ is a fixed small constant.

For the later purpose we need the following two lemmas.

\BL\label{l1.1}
It holds that, for some small $0<\sigma<\F{m}{N-2}(\F{m}{N-2}-\tau)$,
\BEN
\|R\|_{**} \leq  C\left(\F{1}{\mu} \right)^{\F{m }{2}+\sigma}.
\EEN
\EL

\begin{proof}
Define
\[
\Omega_j=\left\{\bar y\in\partial\mathbb R_+^N ~\Big| ~ \bar y=(\bar
y', \bar y'')\in\mathbb R^2\times\mathbb R^{N-3},\ \left\langle
\frac {\bar y'}{|\bar y'|}, \frac{\bar x_j}{|\bar
x_j|}\right\rangle\ge \cos\frac{\pi}{k}\right\}.
\]
We have
\[
\begin{split}
R(\bar y)=&\ K\bigl(\frac {|\bar y|}\mu\bigr)\left(
W_{r,\Lambda}^{2^\#-1}-\sum_{j=1}^k U_{x_j,\Lambda}^{2^\#-1}\right)+\sum_{j=1}^k U_{x_j,\Lambda}^{2^\#-1}\left(
K\bigl(\frac {|\bar y|}\mu\bigr)-1\right)\\
:=&\ J_1+J_2.
\end{split}
\]
From the symmetry, we assume that $\bar y\in\Omega_1$. Then Taylor's theorem gives us
\begin{equation}\label{1-l2-5-3}
|J_1|\le \frac{C}{(1+|\bar y-\bar x_1|)^2}\sum_{j=2}^k \frac1{(1+|\bar y-\bar x_j|)^{N-2}}+C\Bigl(
\sum_{j=2}^k \frac1{(1+|\bar y-\bar x_j|)^{N-2}}\Bigr)^{2^\#-1}.
\end{equation}
Since $|\bar y-\bar x_j|\geq |\bar y-\bar x_1|$ and $|\bar y-\bar x_j|\geq \F{1}{2}|\bar x_j-\bar x_1|$ for $\bar y\in\Omega_1$, we obtain
\begin{equation*}\label{2-l2-5-3}
\begin{split}
&\frac1{(1+|\bar y-\bar x_1|)^2} \frac1{(1+|\bar y-\bar x_j|)^{N-2}} \\
\le &\ C\frac1{(1+|\bar y-\bar x_1|)^2}\frac{1}{(1+|\bar y-\bar x_j|)^{N-2-\alpha}}\frac{1}{(1+|\bar y-\bar x_j|)^{\alpha}}\\
\le&\ C\frac1{|\bar x_j-\bar x_1|^\alpha}\frac1{(1+|\bar y-\bar x_1|)^{N-\alpha} }.  \qquad (0 \leq \alpha \leq N-2)
\end{split}
\end{equation*}
Thus, for any $1 < \alpha \leq N-2$,
\begin{equation}\label{3-l2-5-3}
\frac1{(1+|\bar y-\bar x_1|)^2} \sum_{j=2}^k \frac1{(1+|\bar y-\bar x_j|)^{N-2}}\leq  \frac{C}{(1+|\bar y-\bar x_1|)^{N-\alpha}}\bigl(\frac k\mu\bigr)^\alpha.
\end{equation}
Take $\alpha = \frac{N-2}{2} + \frac{m}{N-2}-\tau \in (1, N-2]$ in
(\ref{3-l2-5-3}), then
\[
\frac1{(1+|\bar y-\bar x_1|)^2} \sum_{j=2}^k \frac1{(1+|\bar y-\bar
x_j|)^{N-2}}\leq  \frac{C}{(1+|\bar y-\bar x_1|)^{\frac{N+2}{2} -
\frac{m}{N-2}+\tau}}\bigl(\frac 1\mu\bigr)^{\frac{m}{2} + \sigma}.
\]

Similarly, for $\bar y\in\Omega_1$ and any $1 < \alpha \leq N-2$, we again have
\[
\sum_{j=2}^k \frac1{(1+|\bar y-\bar x_j|)^{N-2}}\le  \frac{C}{(1+|\bar y-\bar x_1|)^{N-2-\alpha}}\left(\F{k}{\mu}\right)^\alpha.
\]
Now we choose $\alpha=\F{N-2}{N}(\F{N-2}{2}+\F{m}{N-2}-\tau)$. It's
easy to verify that
\[
\alpha-1 > \frac{(N-2)^2+4-2(N-2)\tau-2N}{2N}   \ge 0,
 \]
 and
 \[
\alpha <\F{N-2}{N}(\F{N-2}{2}+\F{m}{N-2})\le\F{N-2}{N}\cdot\F{N}{2}<
N-2
 \]
since $\tau <\F{1}{2(N-2)}$.

Note also that \BEN
\F{Nm\alpha}{(N-2)^2}=\F{m}{2}+\F{m^2}{(N-2)^2}-\F{m\tau}{N-2} \geq
\F{m}{2} +\sigma \EEN owing to $\tau < \F{1}{2(N-2)}$. Thus we can
directly check that
\begin{eqnarray*}
\left(\sum_{j=2}^k \frac1{(1+|\bar y-\bar x_j|)^{N-2}}\right)^{2^\#-1}&=& \frac C{(1+|\bar y-\bar x_1|)^{N-\F{N\alpha}{N-2}}} \left(\F{1}{\mu}\right)^\F{Nm\alpha}{(N-2)^2}\\
&\leq& \frac C{(1+|\bar y-\bar x_1|)^{\F{N+2}{2}-\F{m}{N-2}+\tau}}
\left(\F{1}{\mu}\right)^{\F{m}{2}+\sigma}.
\end{eqnarray*}
The same estimates obviously hold for (\ref{3-l2-5-3}). Thus, we
proved that
\[
\|J_1\|_{**}\le C\left(\frac 1\mu\right)^{\F{m}{2}+\sigma} .
\]

Now, we estimate  $J_2$.  For $\bar y\in \Omega_1$ and $j>1$, similarly $|\bar y-\bar x_j|\geq \F{1}{2}|\bar x_j-\bar x_1|$ indicates that,
for $0\leq\alpha\leq N$,
\[
U_{x_j,\Lambda}^{2^\#-1}(y)\le \frac{C}{(1+|\bar y-\bar x_1|)^{N-\alpha}}\frac{1}{|\bar x_1-\bar x_j|^{\alpha}},
\]
which implies that, for $\alpha=\F{N-2}{2}+\F{m}{N-2}-\tau>1$,
\begin{equation}\label{10-l2-5-3}
\left|\sum_{j=2}^k \Bigl(K\bigl(\frac {|\bar y|}\mu\bigr)-1\Bigr)U_{x_j,\Lambda}^{2^\#-1}\right|
\le \frac C{(1+|\bar y-\bar x_1|)^{\F{N+2}{2}-\F{m}{N-2}+\tau}} \left(\F{1}{\mu}\right)^{\F{m}{2}+\sigma}.
\end{equation}
For $\bar y\in \Omega_1$ and $\left||\bar y|-\mu r_0\right|\ge
\delta \mu$ where $\delta>0$ is a fixed constant, then
\[
\left||\bar y|-|\bar x_1|\right|\ge \left||\bar y|-\mu
r_0\right|-\left||\bar x_1|-\mu r_0\right|\ge \frac12 \delta \mu.
\]
As a result, for any $0\leq\alpha\leq N$,
\begin{align}
\left|U_{x_1,\Lambda}^{2^\#-1}\Bigl(K\bigl(\frac {|\bar y|}\mu\bigr)-1\Bigr)\right|
&\le  \frac{C}{(1+|\bar y-\bar x_1|)^{N-\alpha}}\frac1{\mu^{\alpha}} \nonumber \\
& \leq \F{C}{(1+|\bar y-\bar x_1|)^{N-\F{m}{2}-\sigma}}\left(\F{1}{\mu}\right)^{\F{m}{2}+\sigma} \nonumber\\
& \leq \F{C}{(1+|\bar y-\bar x_1|)^{\F{N+2}{2}-\F{m}{N-2}+\tau}}\left(\F{1}{\mu}\right)^{\F{m}{2}+\sigma}\label{11-l2-5-3}.
\end{align}
If  $\bar y\in \Omega_1$ and $||\bar y|-\mu r_0|\le \delta \mu$, then
\[
\begin{split}
\left|K\bigl(\frac {|\bar y|}\mu\bigr)-1\right| &\le C\left|\frac {|\bar y|}\mu-r_0\right|^m \le  \frac{C}{\mu^m} \Bigl((||\bar y|-|\bar x_1||)^m+||\bar x_1|-\mu r_0|)^{m}\Bigr) \\
&\le  \frac{C}{\mu^m} ||\bar y|-|\bar x_1||^m+\frac{C}{\mu^{m+\bar\theta}}.
\end{split}
\]
and
\[
||\bar y|-|\bar x_1||\le ||\bar y|-\mu r_0|+|\mu  r_0-|\bar x_1||\le 2\delta \mu.
\]
Consequently it holds that, for any $0\leq \alpha\leq m$,
\[
\begin{split}
&\ \frac{ ||\bar y|-|\bar x_1||^m }{\mu^{m}}\frac1{(1+|\bar y-\bar x_1|)^{N}}\\
= &\ \frac1{\mu^{\alpha}}\frac1{(1+|\bar y-\bar x_1|)^{N-\alpha}}
\frac{ ||\bar y|-|\bar x_1||^m }{\mu^{ m -\alpha}}\frac1{(1+|\bar y-\bar x_1|)^{\alpha}}\\
\le &\ \frac C{\mu^{\alpha}}
\frac1{(1+|\bar y-\bar x_1|)^{N-\alpha}}
\frac{ ||\bar y|-|\bar x_1||^{\alpha} }{(1+|\bar y-\bar x_1|)^{\alpha}}\\
\le &\ \frac C{\mu^{\alpha}} \frac1{(1+|\bar y-\bar
x_1|)^{N-\alpha}},
\end{split}
\]
and
\[
\frac{C}{\mu^{m+\bar\theta}} \frac1{(1+|\bar y-\bar x_1|)^{N}} \le \
\frac C{\mu^{\alpha}} \frac1{(1+|\bar y-\bar x_1|)^{N-\alpha}}.
\]

Thus we obtain, for $||\bar y|-\mu r_0|\le \delta \mu$ and  $\alpha
=\F{m}{2}+\sigma$, that
\begin{equation}\label{13-l2-5-3}
\left|U_{x_1,\Lambda}^{2^\#-1}\left( K\left(\frac {|\bar y|}\mu\right)-1\right)\right|
\le  \F{C}{(1+|\bar y-\bar x_1|)^{\F{N+2}{2}-\F{m}{N-2}+\tau}}\left(\F{1}{\mu}\right)^{\F{m}{2}+\sigma}.
\end{equation}
 Combining \eqref{10-l2-5-3}, \eqref{11-l2-5-3} and
 \eqref{13-l2-5-3}, we reach that
\[
\|J_2\|_{**}\leq \left(\frac 1\mu\right)^{\F{m}{2}+\sigma} .
\]
\par

The lemma is concluded.
\end{proof}

\BL\label{l1.2} We have \BEN \|N(\phi)\|_{**}\leq C
\|\phi\|_*^{2^\#-1}. \EEN \EL

\begin{proof}
Obviously, it holds from Taylor's theorem that
\BEN |N(\phi)| \leq
C|\phi|^{2^\#-1} \qquad \text{since }N\geq 5>4.
\EEN

Using the inequality
\[
\sum\limits_{j=1}^k a_jb_j \leq \bigg(\sum\limits_{j=1}^k
a_j^p\bigg)^\frac1p\bigg(\sum\limits_{j=1}^k b_j^q\bigg)^\frac1q \qquad
\text{for } \frac1p + \frac1q =1,\ a_j, b_j \geq 0.
\]
 we have that
\begin{align}
 |N(\phi)| &\le C\|\phi\|_*^{2^\#-1}\left(
\sum_{j=1}^k \frac1{(1+| y-  x_j|)^{\frac{N}2-\F{m}{N-2}+\tau}}\right)^{2^\#-1} \nonumber \\
&\le   C\|\phi\|_*^{2^\#-1}\left(
\sum_{j=1}^k \frac1{(1+|\bar y- \bar x_j|)^{\frac{N}2-\F{m}{N-2}+\tau}}\right)^{2^\#-1} \nonumber \\
&\leq  C\|\phi\|_*^{2^\#-1}  \sum_{j=1}^k \frac1{(1+|\bar y-\bar
x_j|)^{\frac{N+2}{2}-\F{m}{N-2}+ \tau}}\left(\sum_{j=1}^k
\frac1{(1+|\bar y-\bar x_j|)^{\frac{N-2-m}{N-2}+
\tau}}\right)^{\frac{2}{N-2}} \nonumber \\
&\leq  C\|\phi\|_*^{2^\#-1}  \sum_{j=1}^k \frac1{(1+|\bar y-\bar
x_j|)^{\frac{N+2}{2}-\F{m}{N-2}+ \tau}},\label{1-l-1-5-3}
\end{align}
since without loss of generality we may assume that $\bar y \in
\Omega_1$, then
\begin{align*}
 \sum\limits_{j=1}^k \frac1{(1+|\bar y-\bar
x_j|)^{\frac{N-2-m}{N-2}+ \tau}} &\le C + \sum_{j=2}^k \frac1{|\bar
x_1-\bar
x_j|^{\frac{N-2-m}{N-2}+ \tau}} \\
&\le C + \frac{k}{\mu^{\frac{N-2-m}{N-2}+ \tau}} \le C.
\end{align*}

The lemma is concluded.
\end{proof}

\section{Linearized and nonlinear problem}\label{s3}

To solve (\ref{5}), we in this section consider the following
intermediate nonlinear problem \BE\label{lin}
\begin{cases}
\D -\Delta \phi_k =0  & \text{in }\mathbb R_+^N, \\
\D\F{\partial\phi_k}{\partial\nu}- (2^\#-1) K(\F{|y|}{\mu})W_{r,\Lambda}^{2^\#-2}\phi_k= R_k + N(\phi_k) + \sum_{j=1}^2 c_{j}\sum_{i=1}^k U_{x_i,\Lambda}^{2^\#-2}Z_{i,j}  &\text{on }\partial\mathbb R_+^N,\\
\D \phi_k \text{ satisfies (\ref{i}) and (\ref{ii})}, \smallskip \\
\D \left\langle U_{x_i,\Lambda}^{2^\#-2}Z_{i, j},\phi_k \right\rangle=0  \qquad\qquad  i=1,\cdots,k,\; j=1,2,
\end{cases}
\EE for some numbers $c_{j}$, where $\langle
u,v\rangle=\int_{\partial\mathbb R_+^N}uv$ and
\begin{align*}
Z_{i,1}&=\frac{\partial U_{x_i,\Lambda}}{\partial r}=U_{x_i,\Lambda}\F{(N-2)\Lambda^2 (\bar y -\bar x_i)}{(1+\Lambda y_N)^2+\Lambda^2|\bar y-\bar x_i|^2} \cdot \F{\bar x_i}{r}, \\
Z_{i,2}&=\frac{\partial U_{x_i,\Lambda}}{\partial \Lambda}=U_{x_i,\Lambda}\F{N-2}{2\Lambda}\cdot\F{1-\Lambda^2 y_N^2-\Lambda^2|\bar y-\bar x_i|^2}{(1+\Lambda y_N)^2+\Lambda^2|\bar y-\bar x_i|^2}.
\end{align*}

Let us remark that in general we   should also  include  the
translational derivatives of $W_{r, \Lambda}$ in the right hand side
of (\ref{lin}). However due to the symmetry assumption on $\phi$,
this part of kernel
 automatically disappears. This is the main reason for imposing the symmetries (\ref{i}) and (\ref{ii}).

Then the following proposition holds.

\begin{proposition}\label{p1-6-3}
There is an integer $k_0>0$, such that for each $k\ge k_0$, $L_0\le
\Lambda\le L_1$, $|r-\mu r_0|\leq \frac1{\mu^{\bar\theta}}$, where
$\bar\theta>0$ is a fixed small constant, (\ref{lin}) has  a unique
solution $\phi=\phi(r,\Lambda)$, satisfying
\[
\|\phi\|_{*}\le C\bigl(\frac 1\mu\bigr)^{\frac {m}2+\sigma},\qquad
|c_j|\le C\bigl(\frac 1\mu\bigr)^{\frac {m}2+\sigma}, \quad j=1,2.
\]
\end{proposition}

In order to obtain Proposition \ref{p1-6-3}, we first consider the corresponding linearized problem
\BE\label{3.1}
\begin{cases}
\D -\Delta \phi_k =0  & \text{in }\mathbb R_+^N, \\
\D\F{\partial\phi_k}{\partial\nu}- (2^\#-1) K(\F{|y|}{\mu})W_{r,\Lambda}^{2^\#-2}\phi_k= h + \sum_{j=1}^2 c_{j}\sum_{i=1}^k U_{x_i,\Lambda}^{2^\#-2}Z_{i,j}  &\text{on }\partial\mathbb R_+^N,\\
\D \phi_k \text{ satisfies (\ref{i}) and (\ref{ii})}, \smallskip \\
\D \left\langle U_{x_i,\Lambda}^{2^\#-2}Z_{i, j},\phi_k \right\rangle=0  \qquad\qquad  i=1,\cdots,k,\; j=1,2.
\end{cases}
\EE

For any fixed $y=(y_1,\cdots,y_N)\in\mathbb R_+^N$, we denote $G(x,y)$ the Green's function of the problem
\BEN
\begin{cases}
  -\Delta G(x,y) = \delta_y \qquad&\text{for }x\in\mathbb R_+^N, \smallskip \\
  G(x,y)=0 &\text{for }|x|\to\infty, \medskip \\
  \D\F{\partial G}{\partial \nu}(x,y) = 0 & \text{for }x_N=0.
\end{cases}
\EEN It is not difficult to check that \BEN
G(x,y)=\F{1}{\omega_N(N-2)}\left(\F{1}{|x-y|^{N-2}} +
\F{1}{|x-y^s|^{N-2}}\right) \EEN where $\omega_N$ is the volume of
the unit ball in $\mathbb R^N$, $y^s$ is the symmetric point of $y$
with respect to $\partial\mathbb R_+^N = \{x:\ x_N=0\}$, i.e. \BEN
y^s = (\bar y,-y_N). \EEN

\begin{lemma}\label{l1}
Assume that $\phi_k$ solves (\ref{3.1}) for $h=h_{k}$. If $\Vert
h_{k}\Vert_{**}$ goes to zero as $k$ goes to infinity, so does
$\Vert \phi_k\Vert_{*}$.
\end{lemma}

\begin{proof}
We argue by contradiction.  Suppose that there are $k\to +\infty$,
$h=h_{k}$, $\Lambda_k\in [L_0,L_1]$, $r_k\in
[r_0\mu-\frac1{\mu^{\bar\theta}}, r_0\mu
+\frac1{\mu^{\bar\theta}}]$, and $\phi_k$ solving (\ref{3.1}) for
$h=h_{k}$, $\Lambda=\Lambda_k$, $r=r_k$,  with
 $\Vert
h_{k}\Vert_{**}\to 0$, and $\|\phi_k\|_*\ge c'>0$.  We may assume
that $\|\phi_k\|_*=1$. For simplicity, we drop the subscript $k$.\par

First, we estimate  $c_\ell$ ($\ell=1,2$). Multiplying \eqref{3.1} by
$Z_{1,\ell}$ and integrating, we see that $c_j$ satisfies
\begin{align}\label{g6}
&\sum_{j=1}^2\sum_{i=1}^k \bigl\langle  U_{x_i,\Lambda}^{2^\#-2}
Z_{i,j}, Z_{1,\ell} \bigr\rangle c_j \nonumber\\
=& \int_{\partial\mathbb
R^N_+}Z_{1,\ell}\F{\partial\phi}{\partial\nu}-(2^\#-1)\int_{\partial\mathbb
R^N_+} K(\F{|y|}{\mu})W_{r,\Lambda}^{2^\#-2}Z_{1,\ell}\phi
-\int_{\partial\mathbb R_+^N} hZ_{1,\ell}.
\end{align}
By Green's formulas, we have
\begin{align*}
&\ \int_{\partial\mathbb R^N_+}Z_{1,\ell}\F{\partial\phi}{\partial\nu}-(2^\#-1)\int_{\partial\mathbb R^N_+} K(\F{|y|}{\mu})W_{r,\Lambda}^{2^\#-2}Z_{1,\ell}\phi -\int_{\partial\mathbb R_+^N} hZ_{1,\ell} \\
=&\ \int_{\partial\mathbb R_+^N} \phi \left[\F{\partial Z_{1,\ell}}{\partial\nu} - (2^\#-1)K(\F{|y|}{\mu})W_{r,\Lambda}^{2^\#-2}Z_{1,\ell}\right] - \int_{\partial\mathbb R_+^N} hZ_{1,\ell}\\
:=&\ I_1 + I_2.
\end{align*}
The equation of $Z_{1,\ell}$ indicates that, in $\partial\mathbb R_+^N$,
\begin{align*}
&\F{\partial Z_{1,\ell}}{\partial\nu} -
(2^\#-1)K(\F{|y|}{\mu})W_{r,\Lambda}^{2^\#-2}Z_{1,\ell}\\
=&
(2^\#-1)Z_{1,\ell}\left[U_{x_1,\Lambda}^{2^\#-2}-K(\F{|y|}{\mu})W_{r,\Lambda}^{2^\#-2}\right].
\end{align*}
Note that, because of  Lemma \ref{laa1} and Lemma \ref{laa2},
\begin{align*}
&\ \left|\int_{\partial\mathbb R_+^N\setminus\Omega_1}\phi Z_{1,\ell} U_{x_1,\Lambda}^{2^\#-2}\right| \leq C\|\phi\|_*\int_{\partial\mathbb R_+^N\setminus\Omega_1}U_{x_1,\Lambda}^{2^\#-1}
\sum_{i=1}^k\F{1}{(1+|\bar y-\bar x_i|)^{\F{N}{2}-\F{m}{N-2}+\tau}} \\
\leq &\ C\|\phi\|_* \sum_{i=2}^k \int_{\Omega_i} \F{1}{(1+|\bar y-\bar x_1|)^{N}}\F{1}{(1+|\bar y-\bar x_i|)^{\F{N}{2}-\F{m}{N-2}+\tau-1}} \\
\leq&\ C\|\phi\|_* \sum_{i=2}^k\F{1}{|\bar x_i-\bar x_1|^{\F{N}{2}-\F{m}{N-2}}}\int_{\Omega_i} \F{1}{(1+|\bar y-\bar x_i|)^{N-1+\tau}} =o(1)\|\phi\|_*,
\intertext{and}
&\ \left|\int_{\partial\mathbb R_+^N\setminus\Omega_1}\phi Z_{1,\ell}W_{r,\Lambda}^{2^\#-2}\right|  = \left|\sum_{i=2}^k \int_{\Omega_i}\phi Z_{1,\ell}W_{r,\Lambda}^{2^\#-2}\right| \\
\leq&\ C\|\phi\|_*\sum_{i=2}^k \int_{\Omega_i} U_{x_1,\Lambda}W_{r,\Lambda}^{2^\#-2} \sum_{j=1}^k\F{1}{(1+|\bar y-\bar x_j|)^{\F{N}{2}-\F{m}{N-2}+\tau}} \\
\leq &\ C\|\phi\|_* \sum_{i=2}^k \F{1}{|\bar x_i -\bar x_1|^{\F{N+2}{2}-\F{N}{N-2}+\F{2m}{(N-2)^2}}}\int_{\Omega_i} \F{1}{(1+|\bar y-\bar x_i|)^{N-1+\tau}}=o(1)\|\phi\|_*
\end{align*}
since $\F{N+2}{2}-\F{N}{N-2}+\F{2m}{(N-2)^2}>1$ for $N\geq 5$.
Then we have
\begin{align}
I_1 =&\ (2^\#-1) \int_{\Omega_1}\phi Z_{1,\ell}U_{x_1,\Lambda}^{2^\#-2}\left(1-K(\F{|\bar y|}{\mu})\right) \mathrm d\bar y \nonumber \\
&\ +O\left\{\int_{\Omega_1}\phi Z_{1,\ell} \left[U_{x_1,\Lambda}^{2^\#-3}\sum_{i=2}^k U_{x_i,\Lambda}+\big(\sum_{i=2}^k U_{x_i,\Lambda}\big)^{2^\#-2}\right]\mathrm d\bar y \right\} \nonumber\\
&\ +o(1)\|\phi\|_*. \label{6}
\end{align}
Direct computations show that
\begin{align*}
&\ \left|\int_{\Omega_1}\phi Z_{1,\ell}U_{x_1,\Lambda}^{2^\#-2}\left(1-K(\F{|\bar y|}{\mu})\right) \mathrm d\bar y\right| \\
\leq &\ C\F{\|\phi\|_*}{\mu^m}\int_{||\bar y|-\mu r_0|\leq \mu^{\frac{m}{N-2}}}||\bar y|-\mu r_0|^m U_{x_1,\Lambda}^{2^\#-1}\sum_{j=1}^k\F{\mathrm d \bar y}{(1+|\bar y-\bar x_j|)^{\F{N}{2}-\F{m}{N-2}+\tau}}\\
& + C\|\phi\|_* \int_{||\bar y|-\mu r_0|\geq \mu^{\frac{m}{N-2}}}U_{x_1,\Lambda}^{2^\#-1}\sum_{j=1}^k\F{\mathrm d \bar y}{(1+|\bar y-\bar x_j|)^{\F{N}{2}-\F{m}{N-2}+\tau}}\\
\leq &\ C\F{\|\phi\|_*}{\mu^m}\int_{||\bar y|-\mu r_0|\leq \mu^{\frac{m}{N-2}}} \F{||\bar y|-\mu r_0|^m\mathrm d\bar y}{(1+|\bar y-\bar x_1|)^{N+\F{N}{2}-\F{m}{N-2}-\F{N-2-m}{N-2}+\tau}}\\
& +O(\mu^{-\frac{m}{N-2}(\frac{N+2}{2}-\frac{m}{N-2})})\|\phi\|_* \\
=&\ o(1)\|\phi\|_*.
\end{align*}
Similar estimates can be gotten for the second term of (\ref{6}).
Thus we get that \BEN |I_1| = o(1)\|\phi\|_*. \EEN In addition it
holds that, using the estimates in the proof of Lemma \ref{laa1},
\begin{align*}
|I_2| \leq& C\|h\|_{**}\int_{\mathbb R^{N-1}}\F{1}{(1+|\bar y-\bar x_1|)^{N-2}}\sum_{i=1}^k \F{1}{(1+|\bar y-\bar x_i|)^{\F{N+2}{2}-\F{m}{N-2}+\tau}} \\
 \leq& C\|h\|_{**} \int_{\mathbb R^{N-1}}\F{1}{(1+|\bar y-\bar x_1|)^{N-2+\F{N+2}{2}-\F{m}{N-2}+\tau}} \\
& +C\|h\|_{**} \sum_{i=2}^k \int_{\mathbb R^{N-1}} \F{1}{(1+|\bar y-\bar x_1|)^{N-2}}\F{1}{(1+|\bar y-\bar x_i|)^{\F{N+2}{2}-\F{m}{N-2}+\tau}} \\
\leq & C\|h\|_{**}.
\end{align*}
\par

On the other hand, for any $i\neq 1$, it is easy to check that
\BE\label{4}
\left|\left\langle  U_{x_i,\Lambda}^{2^\#-2} Z_{i,j}, Z_{1,\ell} \right\rangle \right|
\leq C\int_{\partial\mathbb R_+^{N}} U_{x_i,\Lambda}^{2^\#-1}U_{x_1,\Lambda} .
\EE
By Lemma \ref{laa1}, we may have that
\begin{align}
&\int_{\partial\mathbb R_+^{N}}
U_{x_i,\Lambda}^{2^\#-1}U_{x_1,\Lambda} \mathrm d\bar y
 \leq \int_{\partial\mathbb R_+^{N}}\F{1}{(1+|\bar y-\bar x_i|)^N}\F{1}{(1+|\bar y-\bar x_1|)^{N-2}} \nonumber\\
\leq &\F{C}{|\bar x_i-\bar x_1|^{N-2}}\int_{\partial\mathbb R_+^{N}}\left[\F{1}{(1+|\bar y-\bar x_i|)^N}+\F{1}{(1+|\bar y-\bar x_1|)^{N}}\right] \nonumber\\
\leq &\F{C}{|\bar x_i-\bar x_1|^{N-2}}. \label{14}
\end{align}
In addition, it is easy to get from the symmetry that, for $j\neq
\ell$, \BE\label{15} \left\langle  U_{x_i,\Lambda}^{2^\#-2} Z_{i,j},
Z_{i,\ell} \right\rangle  = 0. \EE Now we find that the coefficient
matrix of the system (\ref{g6}) with respect to $(c_1,c_2)$ is
nondegenerate. Therefore \BEN |c_\ell| \leq o(1)\|\phi\|_* +
C\|h\|_{**}=o(1). \EEN We claim that \BEN \|\phi\|_{L^\infty(
\overline{\mathbb R_+^N}\cap| y- x_i|\leq R )} = o(1) \qquad
\text{for any }i=1,\ldots,k. \EEN Indeed, by elliptic regularity we
can get a $\hat\phi$ such that $\phi(y-x_i)\to\hat\phi$ in
$C^2_\text{loc}(\overline{\mathbb R_+^N})$ and \BEN
\begin{cases}
\D -\Delta \hat\phi = 0  & \text{in }\mathbb R_+^N, \smallskip \\
\D\F{\partial\hat\phi}{\partial\nu}- (2^\#-1)U_{0,\Lambda}^{2^\#-2}\phi= 0 \quad &\text{on }\partial\mathbb R_+^N, \medskip\\
\D \left\langle U_{0,\Lambda}^{2^\#-2}Z_{0, j},\hat \phi \right\rangle=0  & j=1,2,
\end{cases}
\EEN
This implies $\hat\phi=0$, which concludes the claim.\par

We  next rewrite (\ref{3.1}) as
\begin{align*}
\phi(y)=&\  \int_{\partial\mathbb R_+^N} G(y,z) \sum_{j=1}^2 c_{j}\sum_{i=1}^k U_{x_i,\Lambda}^{2^\#-2}(z)Z_{i,j}(z)\mathrm d\bar z \\
&\ +\int_{\partial\mathbb R_+^N} G(y,z) \left[(2^\#-1)
K(\F{|z|}{\mu})W_{r,\Lambda}^{2^\#-2}(z)\phi(z)+h(z)\right]\mathrm
d\bar z ,
\end{align*}
where $z=(\bar z, 0) \in \mathbb R^{N-1} \times \{0\}.$

Direct computations show that
\begin{align*}
&\ \left|\int_{\partial\mathbb R_+^N} G(y,z) \sum_{i=1}^k U_{x_i,\Lambda}^{2^\#-2}(z)Z_{i,j}(z)\mathrm d \bar z\right|\\
\leq&\ C\sum_{i=1}^k\int_{\partial\mathbb R_+^N} \F{1}{| y-z|^{N-2}}\F{1}{(1+ |z- x_i|)^N}\mathrm d \bar z\\
\leq &\ C\sum_{i=1}^k \F{1}{(1+|\ y- x_i|)^{\frac{N}{2}}} \leq
C\sum_{i=1}^k \F{1}{(1+| y- x_i|)^{\F{N}{2}-\F{m}{N-2}+\tau}},
\end{align*}
where we have used Lemma \ref{laa2}.
\begin{align*}
\left|\int_{\partial\mathbb R_+^N}G(y,z)h(z)\mathrm d\bar z\right|
&\leq C\|h\|_{**} \int_{\partial\mathbb R_+^N} \F{1}{| y-z|^{N-2}}
\sum_{i=1}^k \F{1}{(1+|z- x_i|)^{\F{N+2}{2}-\F{m}{N-2}+\tau}}\mathrm d\bar z \\
&\leq C\|h\|_{**} \sum_{i=1}^k\F{1}{(1+| y-
x_i|)^{\F{N}{2}-\F{m}{N-2}+\tau}} ,
\end{align*}
since $\F{N}{2}-\F{m}{N-2}+\tau < N-2$ and
\begin{align*}
&\ \left|\int_{\partial\mathbb R_+^N} G(y,z) K(\F{|z|}{\mu})W_{r,\Lambda}^{2^\#-2}(z)\phi(z)\mathrm d\bar z \right| \\
\leq&\ o(1)\sum_{i=1}^k \int_{B_R(\bar x_i)}\F{1}{| y-  z|^{N-2}}W_{r,\Lambda}^{2^\#-2}( z) \mathrm d\bar z\\
&\ +\|\phi\|_* \int_{\partial\mathbb R_+^N\setminus \underset{i=1}{\overset{k}\cup} B_R(\bar x_i)} \F{1}{| y-  z|^{N-2}}W_{r,\Lambda}^{2^\#-2}( z) \sum_{j=1}^k\frac1{(1+| z- x_j|)^{\F{N}{2}-\F{m}{N-2}+\tau}}\mathrm d\bar z\\
\leq&\  o(1)\sum_{i=1}^k \F{1}{(1+| y- x_i|)^{\F{N}{2}-\F{m}{N-2}+\tau}} \\
&\qquad +\F{C}{R^{\F{(m-2)N+4}{(N-2)^2}}}\|\phi\|_* \sum_{i=1}^k \int_{\Omega_i} \F{1}{| y -z|^{N-2}}\F{1}{(1+| z- x_i|)^{\F{N}{2}-\F{m}{N-2}+1+\tau}}\mathrm d\bar z\\
\leq&\
\left(o(1)+\F{C}{R^{\F{(m-2)N+4}{(N-2)^2}}}\|\phi\|_*\right)\sum_{i=1}^k
\F{1}{(1+| y- x_i|)^{\F{N}{2}-\F{m}{N-2}+\tau}}  .
\end{align*}
Up to now, choosing $R$ large, we obtain that \BEN \|\phi\|_* \leq C
\|h\|_{**} +o(1) = o(1), \EEN a contradiction.
\end{proof}

From Lemma~\ref{l1}, using the same argument as in the proof of
Proposition~4.1 in \cite{DFM}, we can prove the following result.

\begin{proposition}\label{p1}
There exists $k_0>0 $ and a constant $C>0$, independent of $k$, such
that for all $k\ge k_0$ and all $h\in L^{\infty}(\mathbb R^{N-1})$,
problem $(\ref{3.1})$ has a unique solution $\phi:= L_k(h)$.
Besides,
\begin{equation}\label{Le}
\Vert L_k(h)\Vert_*\leq C\Vert h\Vert_{**},\qquad
|c_{\ell}|\leq C\Vert h\Vert_{**}.
\end{equation}
\end{proposition}

It is now ready for us to prove Proposition \ref{p1-6-3}.
\begin{proof}[Proof of Proposition~\ref{p1-6-3}]
Let us recall that $ \mu = k^{\frac{N-2}{N-2-m}}$ and denote
\begin{multline*}
E=\bigg\{ \phi\in C(\overline{\mathbb R_+^N} ):\  \phi \text{ satisfies }(\ref{i})\text{ and }(\ref{ii}), \ \|\phi\|_*\leq \eta\bigl(\frac 1\mu\bigr)^{\frac {m}{2}+\sigma}, \\
\int_{\partial\mathbb R_+^N} U_{x_i,\Lambda}^{2^\#-2} Z_{i,j}\phi=0\text{ for any } i=1,\cdots,k,\  j=1,2
\bigg\},
\end{multline*}
where $\eta>0$ is a fixed large constant.  Then (\ref{5}) is
equivalent to
\[
\phi= A(\phi):= L(N(\phi))+L(R).
\]
We will first prove that $A$ is a contraction mapping from $E$ to
$E$.
\par

In fact, by Lemma \ref{l1.1} and Lemma \ref{l1.2} we have
\begin{align*}
\|\phi\|_* &\leq C\|R\|_{**} + C\|N(\phi)\|_{**} \\
&\leq C\left(\F{1}{\mu}\right)^{\F{m}{2}+\sigma}
+ C \eta^\F{N}{N-2} \left(\F{1}{\mu}\right)^{\left(\F{m}{2}+\sigma\right)\F{N}{N-2}}\\
&\leq C\left(\F{1}{\mu}\right)^{\F{m}{2}+\sigma}
\left(1+\eta^\F{N}{N-2}\left(\F{1}{\mu}\right)^{\left(\F{m}{2}+\sigma\right)\F{2}{N-2}}\right)\\
&\leq \eta \left(\F{1}{\mu}\right)^{\F{m}{2}+\sigma}.
\end{align*}
\newpage

Thus $A$ maps $E$ to $E$ itself.\par

On the other hand, it holds obviously that
\[
\| A(\phi_1)-A(\phi_2)\|_{*}= \|L(N(\phi_1))-L(N(\phi_2))\|_{*}
\le C \|N(\phi_1)-N(\phi_2)\|_{**}.
\]
Since $2^{\#}-2 < 1$, we have that
\[
|N'(t)| \le C|t|^{2^\#-2}.
\]
Thus for any $y \in \partial\mathbb{R}_+^N$,
\begin{align*}
&C|N(\phi_1)-N(\phi_2)| \leq C(|\phi_1|^{2^\#-2}+|\phi_2|^{2^\#-2})|\phi_1-\phi_2| \\
\leq & C(\|\phi_1\|_*^{2^\#-2} + \|\phi_2\|_*^{2^\#-2})\|\phi_1-\phi_2\|_* \left(\sum_{j=1}^k \F{1}{(1+|\bar y- \bar x_j|)^{\F{N}{2}-\F{m}{N-2}+\tau}}\right)^{2^\#-1}\\
\leq & C\eta^\F{2}{N-2}\bigl(\frac 1\mu\bigr)^{\left(\F{m}{2}+\sigma\right)\F{2}{N-2}} \|\phi_1-\phi_2\|_* \sum_{j=1}^k \frac1{(1+|\bar y- \bar x_j|)^{\frac{N+2}{2}-\F{m}{N-2}+ \tau}}\\
\leq & \F{1}{2}\|\phi_1-\phi_2\|_* \sum_{j=1}^k \frac1{(1+|\bar
y-\bar x_j|)^{\frac{N+2}{2}-\F{m}{N-2}+ \tau}}.
\end{align*}
where the relation (\ref{1-l-1-5-3}) has been used.  Thus $A$ is a
contraction mapping.\par

It follows from the contraction mapping theorem that there is a
unique $\phi\in E$ such that
\[
\phi=A(\phi).
\]
The proof finishes.
\end{proof}

\section{Variational reduction}\label{s4}

After problem (\ref{lin}) has been solved, we find a solution to
problem (\ref{5}) and hence to the original problem (\ref{8}) if
$(r,\Lambda)$ is such that \BEN c_{j}(r,\Lambda)=0,\qquad j=1, 2.
\EEN This problem is in fact variational.\par

Let
\BEN
F(r,\Lambda)=I(W_{r,\Lambda}+\phi)
\EEN
where $\phi$ is the function obtained in Proposition \ref{p1-6-3} and
\[
I(u)=\frac{1}{2}\int_{\mathbb R_+^N} |Du|^2-\frac1{2^\#}\int_{\partial\mathbb R_+^N} K\bigl(\frac{|y|}\mu\bigr)|u|^{2^\#}.
\]

\BP
Assume $(r,\Lambda)$ is a critical point of $F(r,\Lambda)$. Then $c_j=0$ for any $j=1,2$.
\EP

\begin{proof}
By (\ref{14}) and (\ref{15}), we first get that
\begin{align*}
&\ \sum_{j=1}^2 c_j \sum_{i=1}^k\int_{\partial\mathbb R_+^N} U_{x_i,\Lambda}^{2^\#-2} Z_{i,j}\F{\partial W_{r,\Lambda}}{\partial \Lambda} = \sum_{j=1}^2  c_j \sum_{i=1}^k\sum_{\ell=1}^k \int_{\partial\mathbb R_+^N} U_{x_i,\Lambda}^{2^\#-2} Z_{i,j} Z_{\ell,2}\\
=&\ \sum_{j=1}^2 c_j \sum_{i=1}^k \int_{\partial\mathbb R_+^N} U_{x_i,\Lambda}^{2^\#-2} Z_{i,j} Z_{i,2} + O\left(\sum_{j=1}^2 c_j \sum_{i=1}^k\sum_{\ell\neq i}\F{1}{|\bar x_i-\bar x_\ell|^{N-2}}\right) \\
=&\ c_2\sum_{i=1}^k \int_{\partial\mathbb R_+^N} U_{x_i,\Lambda}^{2^\#-2} Z_{i,2}^2 + O\left(k\mu^{-m}\sum_{j=1}^2 c_j\right)\\
=&\ k c_2\left[ \int_{\partial\mathbb R_+^N}
U_{x_1,\Lambda}^{2^\#-2} Z_{1,2}^2 +
O\left(\F{1}{\mu}\right)^{m}\right]  +k O\left(\F{1}{\mu}\right)^{m}
c_1,
\end{align*}
and similarly
\begin{align*}
&\sum_{j=1}^2 c_j \sum_{i=1}^k\int_{\partial\mathbb R_+^N} U_{x_i,\Lambda}^{2^\#-2} Z_{i,j}\F{\partial W_{r,\Lambda}}{\partial r}\\
=&kc_1\left[ \int_{\partial\mathbb R_+^N} U_{x_1,\Lambda}^{2^\#-2}
Z_{1,1}^2 + O\left(\F{1}{\mu}\right)^{m}\right]  +
kO\left(\F{1}{\mu}\right)^{m} c_2.
\end{align*}
In addition, since
\begin{align*}
\left|\F{\partial Z_{i,1}}{\partial \Lambda}\right| & \leq CU_{x_i,\Lambda}\F{\Lambda^2 |\bar y -\bar x_i|}{(1+\Lambda y_N)^2+\Lambda^2|\bar y-\bar x_i|^2}\leq CU_{x_i,\Lambda},\\
\left|\F{\partial Z_{i,2}}{\partial \Lambda}\right| & \leq CU_{x_i,\Lambda},
\end{align*}
it holds that
\begin{align}
&\ \left|\sum_{i=1}^k\int_{\partial\mathbb R_+^N} U_{x_i,\Lambda}^{2^\#-2} Z_{i,j}\F{\partial \phi}{\partial \Lambda}\right| =\left|-\sum_{i=1}^k\int_{\partial\mathbb R_+^N} \F{\partial \left(U_{x_i,\Lambda}^{2^\#-2} Z_{i,j}\right)}{\partial \Lambda} \phi \right| \nonumber \\
\leq &\  C\sum_{i=1}^k\int_{\partial\mathbb R_+^N}U_{x_i,\Lambda}^{2^\#-1}|\phi| \nonumber\\
 \leq &\ C \|\phi\|_*\int_{\partial\mathbb R_+^N} \sum_{i=1}^k\F{1}{(1 + |\bar y-\bar x_i|)^N} \sum_{\ell=1}^k \F{1}{(1+|\bar y-\bar x_\ell|)^{\F{N}{2}-\F{m}{N-2}+\tau}}\mathrm d\bar y\nonumber\\
\leq &\ C\|\phi\|_* \sum_{i=1}^k\int_{\Omega_i} \F{\mathrm d \bar y}{(1+|\bar y-\bar x_i|)^{N-1+\F{N}{2}-\F{m}{N-2}-\F{2(N-2-m)}{N-2}+\tau}} \nonumber \\
\leq &\ C k \left(\F{1}{\mu}\right)^{\F{m}{2}+\sigma}. \label{16}
\end{align}
A same estimate also holds for  $\sum_{i=1}^k\int_{\partial\mathbb R_+^N} U_{x_i,\Lambda}^{2^\#-2} Z_{i,j}\F{\partial \phi}{\partial r}$.\par

Finally we note that
\begin{align*}
0=\F{\partial F}{\partial r} &= \int_{\partial\mathbb R_+^N} \sum_{j=1}^2 c_{j}\sum_{i=1}^k U_{x_i,\Lambda}^{2^\#-2}Z_{i,j}\left(\F{\partial W_{r,\Lambda}}{\partial r}+\F{\partial \phi}{\partial r}\right),\\
0=\F{\partial F}{\partial \Lambda} &= \int_{\partial\mathbb R_+^N} \sum_{j=1}^2 c_{j}\sum_{i=1}^k U_{x_i,\Lambda}^{2^\#-2}Z_{i,j}\left(\F{\partial W_{r,\Lambda}}{\partial \Lambda}+\F{\partial \phi}{\partial \Lambda}\right).
\end{align*}
Therefore it is easy for us to get that $c_j=0$ ($j=1,2$) from the
nondegeneracy of their coefficient matrix.
\end{proof}

\begin{proposition}\label{p2-6-3}
We have
\[
\begin{split}
F(r,\Lambda)= & I(W_{r,\Lambda})+O\Bigl(\frac k{\mu^{m+2\sigma}}\Bigr)\\
=& k\Bigl( A +\frac{B_1}{\Lambda^{m}\mu^m}
+\frac{B_2}{\Lambda^{m-2}\mu^m}
(\mu r_0 -|x_1|))^2\\
&-\sum_{i=2}^k\frac{ B_3 }{\Lambda^{N-2}|x_1-x_j|^{N-2}}+
O\Bigl(\frac1{\mu^{m+\sigma}}+\frac{1}{\mu^m} |\mu r_0
-|x_1||^3\Bigr) \Bigr),
\end{split}
\]
where $B_i>0$, $i=1,2,3$, are some
constants.
\end{proposition}

\begin{proof}
Since
\[
D I\bigl(  W_{r,\Lambda} \bigr)\phi=0,
\]
there is $t\in (0,1)$ such that
\[
\begin{split}
&\ F(r,\Lambda)= I(  W_{r,\Lambda}) +\frac12 D^2I\bigl(  W_{r,\Lambda} +t \phi\bigr) (\phi,\phi)\\
=&\ I(  W_{r,\Lambda}) + \F{1}{2}\int_{\mathbb R_+^N} |D \phi|^2-\F{2^\#-1}{2}\int_{\partial\mathbb R_+^N}K(\frac {|y|}\mu)
\bigl( W_{r,\Lambda}+t \phi\bigr)^{2^\#-2}\phi^2\\
=&\ I(  W_{r,\Lambda}) - \F{2^\#-1}{2}\int_{\partial\mathbb R_+^N}
K\bigl(\frac {|y|}\mu\bigr) \left[\bigl( W_r +t \phi\bigr)^{2^\#-2}-W_r^{2^\#-2}\right]\phi^2\\
&+\int_{\partial\mathbb R_+^N}\bigl( N(\phi)-R\bigr)\phi\\
=&\ I(  W_{r,\Lambda})+O\Bigl(\int_{\partial\mathbb R_+^N}\bigl(|\phi|^{2^\#}+|N(\phi)||\phi|+|R||\phi|\bigr)\Bigr).
\end{split}
\]
Moreover it is easy to check that,
\begin{align*}
&\int_{\partial\mathbb R_+^N}|N(\phi)| |\phi|\\
\leq&  C\|N(\phi)\|_{**}\|\phi\|_*
 \int_{\partial\mathbb R^N_+} \sum_{i,j=1}^k\frac1{(1+|\bar y-\bar x_j|)^{\frac{N+2}2-\F{m}{N-2}+\tau}}\frac1{(1+|\bar y-\bar x_i|)^{\frac{N}2-\F{m}{N-2}+\tau}}\mathrm d\bar y \\
 \leq &\ C (\F{1}{\mu})^{m+2\sigma}  \sum_{i=1}^k \int_{\partial\mathbb R_+^N}\F{\mathrm d\bar y}{(1+|\bar y-\bar x_i|)^{N+1-\F{2m}{N-2}+2\tau-\frac{N-2-m}{N-2}}}\\
 \leq &\ Ck\left(\F{1}{\mu}\right)^{m+2\sigma}.
\end{align*}
So does $\int_{\partial\mathbb R_+^N}|R| |\phi|$. Similarly, we have
\begin{align}
\int_{\partial\mathbb R_+^N}|\phi|^{2^\#} &\leq \|\phi\|_*^{2^\#}\int_{\partial\mathbb R_+^N} \left(\sum_{i=1}^k \frac1{(1+|\bar y-\bar x_i|)^{\frac{N}2-\F{m}{N-2}+\tau}}\right)^{2^\#}\mathrm  d\bar y\nonumber \\
&\leq C \|\phi\|_*^{2^\#} \sum_{\ell=1}^k \int_{\Omega_\ell} \F{\mathrm d\bar y}{(1+|\bar y-\bar x_\ell|)^{\left(\frac{N}2-\F{m}{N-2}-\F{N-2-m}{N-2}+\tau\right)\F{2(N-1)}{N-2}}} \nonumber\\
&\leq C \|\phi\|_*^{2^\#} \sum_{\ell=1}^k \int_{\Omega_\ell} \F{\mathrm d\bar y}{(1+|\bar y-\bar x_\ell|)^{N-1+\tau}} \nonumber\\
&\leq Ck\|\phi\|_*^{2^\#}  \leq C k \left(\F{1}{\mu}\right)^{\F{m(N-1)}{N-2}+2^\#\sigma}. \label{13}
\end{align}
From Proposition \ref{pa2} we conclude the proof.
\end{proof}

\begin{proposition}\label{p1-7-3}
We have
\[
\begin{split}
&\frac{\partial F(r,\Lambda)}{\partial \Lambda}\\
=& k\left( -\frac{B_1 m }{\Lambda^{m+1}\mu^m} +\sum_{i=2}^k\frac{
B_3(N-2) }{\Lambda^{N-1}|x_1-x_j|^{N-2}} +
O\Bigl(\frac1{\mu^{m+\sigma}}+\frac{1}{\mu^m} |\mu r_0
-|x_1||^2\Bigr) \right).
\end{split}
\]
\end{proposition}

\begin{proof}
First we note from (\ref{16}) and Proposition \ref{p1-6-3} that
\begin{align}
&\ \frac{\partial F(r,\Lambda)}{\partial \Lambda}
= D I(W_{r,\Lambda}+\phi) \left(\frac{\partial W_{r,\Lambda}}{\partial \Lambda}+\frac{\partial \phi}{\partial \Lambda}\right) \nonumber \\
=&\ DI(W_{r,\Lambda}+\phi)\left( \frac{\partial W_{r,\Lambda}}{\partial \Lambda}
\right) +\sum_{j=1}^2\sum_{i=1}^kc_j\left\langle U^{2^\#-2}_{x_i,\Lambda} Z_{i,j},
\frac{\partial \phi}{\partial \Lambda}
\right\rangle \nonumber  \\
=&\ DI(W_{r,\Lambda}+\phi)\left( \frac{\partial W_{r,\Lambda}}{\partial \Lambda}
\right) + O\left(k\mu^{-m-\sigma}\right)\nonumber  \\
=&\ \F{\partial}{\partial\Lambda}I(W_{r,\Lambda})-\int_{\partial\mathbb R_+^N}K(\F{|y|}{\mu})\left[(W_{r,\Lambda}+\phi)^{2^\#-1}-W_{r,\Lambda}^{2^\#-1}\right]\frac{\partial W_{r,\Lambda}}{\partial \Lambda} \nonumber  \\
& + O\left(k\mu^{-m-\sigma}\right), \label{17}
\end{align}
because the orthogonality of $\phi$ implies
\BEN
\int_{\mathbb R_+^N} \nabla\phi\nabla \frac{\partial W_{r,\Lambda}}{\partial \Lambda}
=-\int_{\mathbb R_+^N}\phi\Delta\frac{\partial W_{r,\Lambda}}{\partial \Lambda}+\int_{\partial\mathbb R_+^N}\phi\F{\partial}{\partial\nu}\left(\frac{\partial W_{r,\Lambda}}{\partial \Lambda}\right)
=0.
\EEN
Next we will deal with the second term in the right side of (\ref{17}). It holds that
\begin{align*}
&\ \int_{\partial\mathbb R_+^N}K(\F{|y|}{\mu})\left[(W_{r,\Lambda}+\phi)^{2^\#-1}-W_{r,\Lambda}^{2^\#-1}\right]\frac{\partial W_{r,\Lambda}}{\partial \Lambda} \\
=&\ (2^\#-1)\int_{\partial\mathbb R_+^N}K(\F{|y|}{\mu})W_{r,\Lambda}^{2^\#-2}\frac{\partial W_{r,\Lambda}}{\partial \Lambda}\phi
+O\left(\int_{\partial\mathbb R_+^N}  W_{r,\Lambda}^{2^\#-2} |\phi|^{2} + |\phi|^{2^\#} \right) .
\end{align*}
For $\alpha=\F{N-2-m}{N-2}$, we know that in $\Omega_i$,
 \BEN
\sum_{j\neq i}\F{1}{(1+|\bar y-\bar x_j|)^{N-2}} \leq \F{1}{(1+|\bar
y-\bar x_i|)^{N-2-\alpha}} \sum_{j\neq i}\F{1}{|\bar x_j-\bar
x_i|^\alpha}, \EEN which leads to
\begin{gather*}
W_{r,\Lambda}^{2^\#-2} \leq \F{C}{(1+|\bar y-\bar x_i|)^{2-\F{2\alpha}{N-2}}}, \\
\sum_{j=1}^k\frac1{(1+|\bar y-\bar x_j|)^{\F{N}{2}-\F{m}{N-2}+\tau}}
\leq \F{C}{(1+|\bar y-\bar x_i|)^{\F{N}{2}-\F{m}{N-2}+\tau-\alpha}}.
\end{gather*}
As a result, we find that

\begin{align*}
\int_{\partial\mathbb R_+^N}  W_{r,\Lambda}^{2^\#-2} |\phi|^{2} &\leq C\|\phi\|^2_* \sum_{i=1}^k \int_{\Omega_i}\F{\mathrm d\bar y}{(1+|\bar y-\bar x_i|)^{2-\F{2\alpha}{N-2}+N-\F{2m}{N-2}+2\tau-2\alpha}}\\
&\leq C\|\phi\|^2_*\sum_{i=1}^k\int_{\Omega_i}\F{\mathrm d\bar y}{(1+|\bar y-\bar x_i|)^{N-1+\F{N^2-6N+2m+8}{(N-2)^2}+2\tau}} \\
&\leq Ck\mu^{-m-2\sigma}.
\end{align*}
A similar estimate also holds for $\int_{\partial\mathbb R_+^N}|\phi|^{2^\#}$ which is given by (\ref{13}).
Furthermore, from the orthogonality of $\phi$, we have that
\begin{align*}
&\ \int_{\partial\mathbb R_+^N}K(\F{|y|}{\mu})W_{r,\Lambda}^{2^\#-2}\frac{\partial W_{r,\Lambda}}{\partial \Lambda}\phi \\
=&\ \int_{\partial\mathbb R_+^N} K(\F{|y|}{\mu}) \left(W_{r,\Lambda}^{2^\#-2}\frac{\partial W_{r,\Lambda}}{\partial \Lambda} - \sum_{i=1}^k U_{x_i,\Lambda}^{2^\#-2} \F{\partial U_{x_i,\Lambda}}{\partial\Lambda}\right)\phi \\
& + \sum_{i=1}^k\int_{\partial\mathbb R_+^N} \left[K(\F{|y|}{\mu})-1\right] U_{x_i,\Lambda}^{2^\#-2} \F{\partial U_{x_i,\Lambda}}{\partial\Lambda}\phi\\
=&\ k\int_{\Omega_1} K(\F{|\bar y|}{\mu}) \left(W_{r,\Lambda}^{2^\#-2}\frac{\partial W_{r,\Lambda}}{\partial \Lambda} - \sum_{i=1}^k U_{x_i,\Lambda}^{2^\#-2} \F{\partial U_{x_i,\Lambda}}{\partial\Lambda}\right)\phi \\
& + k \int_{\partial\mathbb R_+^N} \left[K(\F{|y|}{\mu})-1\right]
U_{x_1,\Lambda}^{2^\#-2} \F{\partial
U_{x_1,\Lambda}}{\partial\Lambda}\phi.
\end{align*}
Thus we can check that
\begin{align*}
&\ \left|\int_{\Omega_1} K(\F{|\bar y|}{\mu}) \left(W_{r,\Lambda}^{2^\#-2}\frac{\partial W_{r,\Lambda}}{\partial \Lambda} - \sum_{i=1}^k U_{x_i,\Lambda}^{2^\#-2} \F{\partial U_{x_i,\Lambda}}{\partial\Lambda}\right)\phi\right| \\
\leq &\ C\int_{\Omega_1} \left( U_{x_1,\Lambda}^{2^\#-2}\sum_{i=2}^k U_{x_i,\Lambda}+\sum_{i=2}^k U_{x_i,\Lambda}^{2^\#-1}\right)|\phi|\\
\leq &\ C\left(\F{1}{\mu}\right)^{\F{Nm}{2(N-2)}+\F{m}{2}+\sigma} \leq C\left(\F{1}{\mu}\right)^{m+\sigma},
\end{align*}
and, using Lemma \ref{laa1},
\begin{align*}
&\ \left|\int_{\partial\mathbb R_+^N} \left[K(\F{| y|}{\mu})-1\right] U_{x_1,\Lambda}^{2^\#-2} \F{\partial U_{x_1,\Lambda}}{\partial\Lambda}\phi\right| \\
\leq &\ C\|\phi\|_*\F{1}{\mu^m}\int_{||\bar y|-\mu r_0|\leq \mu^{\frac{m}{N-2}}} \F{||\bar y|-\mu r_0|^m}{(1+|\bar y-\bar x_1|)^N}\sum_{i=1}^k \F{1}{(1+|\bar y-\bar x_i|)^{\F{N}{2}-\F{m}{N-2}+\tau}}\\
&+ C\|\phi\|_*\int_{||\bar y|-\mu r_0|\geq \mu^{\frac{m}{N-2}}}\F{1}{(1+|\bar y-\bar x_1|)^N}\sum_{i=1}^k \F{1}{(1+|\bar y-\bar x_i|)^{\F{N}{2}-\F{m}{N-2}+\tau}}\\
\leq & \F{C}{\mu^{m+\sigma}}.
\end{align*}
Thus we finish the proof from Proposition \ref{pa3}.
\end{proof}

\section{Proof of Theorem \ref{t2}} \label{s5}

Since
\[
|x_j-x_1|=2|x_1|\sin\frac{(j-1)\pi}k, \quad j=2,\dots, k,
\]
we have
\[
\begin{split}
&\sum_{j=2}^k \frac1{|x_j-x_1|^{N-2}}=\frac1{(2|x_1|)^{N-2}}\sum_{j=2}^k \frac1{(\sin\frac{(j-1)\pi}k)^{N-2}}\\
=&
\begin{cases}
\frac2{(2|x_1|)^{N-2}}\sum_{j=2}^{\frac k2}
\frac1{(\sin\frac{(j-1)\pi}k)^{N-2}} +\frac1{(2|x_1|)^{N-2}},
& \text{if $k$ is even};\\
\frac2{(2|x_1|)^{N-2}}\sum_{j=2}^{[\frac k2]}
\frac1{(\sin\frac{(j-1)\pi}k)^{N-2}}, &  \text{if $k$ is old}.
\end{cases}
\end{split}
\]
But
\[
0<c'\le \frac{\sin\frac{(j-1)\pi}k}{\frac{(j-1)\pi}k}\le c'',  \quad j=2,\cdots, \left[\frac k2\right].
\]
So, there is a constant $B_4>0$, such that
\[
\sum_{j=2}^k \frac1{|x_j-x_1|^{N-2}}=\frac{B_4k^{N-2}}{|x_1|^{N-2}}+O\Bigl(\frac k{|x_1|^{N-2}}\Bigr).
\]
Thus, we obtain
\[
\begin{split}
F(r,\Lambda)
=& k\Bigl( A +\frac{B_1}{\Lambda^{m}\mu^m}
+\frac{B_2}{\Lambda^{m-2}\mu^m}
(\mu r_0 -r)^2\\
&\quad-\frac{ B_3B_4 k^{N-2}}{\Lambda^{N-2}r^{N-2}}+
O\Bigl(\frac1{\mu^{m+\sigma}}+\frac{1}{\mu^m} |\mu r_0 -r|^3+\frac
k{r^{N-2}}\Bigr) \Bigr),
\end{split}
\]
and
\[
\begin{split}
&\frac{\partial F(r,\Lambda)}{\partial \Lambda}\\
=& k\left( -\frac{B_1 m }{\Lambda^{m+1}\mu^m} +\frac{ B_3B_4(N-2)
k^{N-2}}{\Lambda^{N-1}r^{N-2}} +
O\Bigl(\frac1{\mu^{m+\sigma}}+\frac{1}{\mu^m} |\mu r_0 -r|^2+\frac
k{r^{N-2}}\Bigr) \right).
\end{split}
\]\par

Let $\Lambda_0$ be the solution of
\[
-\frac{B_1 m }{\Lambda^{m+1}} +\frac{ B_3B_4(N-2)
}{\Lambda^{N-1}r_0^{N-2}}=0,
\]
that is
\[
\Lambda_0= \Bigl(\frac {B_3B_4(N-2)}{B_1 m
r_0^{N-2}}\Bigr)^{\frac1{N-2-m}}.
\]
Define
\[
D=\left\{ (r,\Lambda):  r\in \left[\mu
r_0-\frac1{\mu^{\bar\theta}},\mu
r_0+\frac1{\mu^{\bar\theta}}\right],\;\; \Lambda\in \left[\Lambda_0
-\frac1{\mu^{\frac32\bar\theta}},\Lambda_0
+\frac1{\mu^{\frac32\bar\theta}}\right] \right\},
\]
where $\bar\theta>0$ is a small constant.\par

\textit{The existence of a critical point  in D of $F(r,\Lambda)$
may be identically proved just as \cite[Prop. 3.3, Prop. 3.4]{WY}}.
We omit the details.\par

It remains to prove that the solution we found
$v_\mu=W_{r,\Lambda}+\phi$ is positive. Testing the equation to
$v_\mu$ (\ref{8}) against $v_\mu^-=\min\{v_\mu,0\}$ itself, it holds
that \BEN \int_{\mathbb R_+^N} |\nabla v_\mu^-|^2 =
\int_{\partial\mathbb R_+^N} K\big(\F{|y|}{\mu}\big)
(v_\mu^-)^{2^\#}. \EEN Moreover the trace theorem tells us that \BEN
\int_{\partial\mathbb R_+^N} K\big(\F{|y|}{\mu}\big)
(v_\mu^-)^{2^\#} \leq C \left(\int_{\mathbb R_+^N} |\nabla
v_\mu^-|^2 \right)^\F{2^\#}{2}. \EEN
 Combining the above two
inequalities, we get that \BE\label{9} \int_{\partial\mathbb R_+^N}
K\big(\F{|y|}{\mu}\big)(v_\mu^-)^{2^\#}  \geq C \qquad \text{or}
\qquad v_\mu^- \equiv 0 \quad \text{on }\partial\mathbb R_+^N.
\EE\par

On the other hand, we know that $|v_\mu^-|\leq |\phi|$ since
$W_{r,\Lambda}>0$. Thus, by (\ref{13}) it holds that \BEN
\int_{\partial\mathbb R_+^N} K\big(\F{|y|}{\mu}\big)(v_\mu^-)^{2^\#}
\leq C\int_{\partial\mathbb R_+^N}|\phi|^{2^\#} \leq
C\left(\F{1}{\mu}\right)^{2^\#(\frac{m}{2}+\sigma)} =o(1). \EEN On
account of (\ref{9}) again it must hold that $v_\mu^-\equiv 0$ on
$\partial\mathbb R_+^N$, which implies that \BEN v_\mu \geq 0 \qquad
\text{on }\partial\mathbb R_+^N. \EEN Therefore $v_\mu$ must be
positive because it is harmonic in $\mathbb R_+^N$.

\section{Appendix}\label{s6}

In all of the appendixes, we always assume that
\[
x_j=\left(r \cos\frac{2(j-1)\pi}k, r\sin\frac{2(j-1)\pi}k,0\right)\qquad j=1,\cdots,k,
\]
where $0$ is the zero vector in $\mathbb R^{N-2}$, and $r\in \left[r_0 \mu -\frac1{\mu^{\bar\theta}},r_0 \mu+\frac1{\mu^{\bar\theta}}\right]$ for
some small $\bar\theta>0$.

\subsection{Energy expansion of the approximate solution}

In this section, we will calculate $I(W_{r,\Lambda})$.\par

Let us recall that
\begin{align*}
\mu &= k^{\frac{N-2}{N-2-m}},\\
I(u) &=\frac{1}{2}\int_{\mathbb R_+^N} |Du|^2-\frac1{2^\#}\int_{\partial\mathbb R_+^N} K\bigl(\frac{|y|}\mu\bigr)|u|^{2^\#},\\
U_{x_j,\Lambda}(y)&=\bigl(N-2\bigr)^\frac{N-2}{2}\left[\frac{\Lambda }{(1+\Lambda y_N)^2+\Lambda^2|\bar y-\bar x_j|^2}\right]^\frac{N-2}{2}
\intertext{and}
W_{r,\Lambda}(y)&=(N-2)^\frac{N-2}{2}\sum_{j=1}^k \left[\frac{\Lambda }{(1+\Lambda y_N)^2+\Lambda^2|\bar y-\bar x_j|^2}\right]^\frac{N-2}{2}.
\end{align*}

\begin{proposition}\label{pa2}
 We have
\begin{align*}
I(W_{r,\Lambda})=& k\bigg[ A +\frac{B_1}{\Lambda^{m}\mu^m}
+\frac{B_2}{\Lambda^{m-2}\mu^m}
(\mu r_0 -r)^2\\
&-\sum_{i=2}^k\frac{ B_3 }{\Lambda^{N-2}|x_1-x_j|^{N-2}}+
O\Bigl(\frac1{\mu^{m+\sigma}}+\F{1}{\mu^m}|\mu r_0-r|^{2+\tilde\sigma}\Bigr) \bigg],
\end{align*}
where $A$, $B_i$ ($i=1,2,3$) are some positive constants only depending on $N$, $r=|x_1|$ and $\tilde\sigma=\min\{m-2,1\}$.
\end{proposition}

\begin{proof}
First let us calculate $\int_{\mathbb R^N } |D W_{r,\Lambda}|^2$. It is easy to get that, for $j=1,\cdots,k$,
\BE\label{a.1}
A_{N}:=\int_{\partial\mathbb R^N_+ } U_{x_j,\Lambda}^{2^\#} =(N-2)^{N-1} \int_{\mathbb R^{N-1}} \F{\mathrm dz}{(1+|z|^2)^{N-1}}.
\EE
By using the symmetry, we claim that
\begin{align}\label{a.2}
&\sum_{\substack{i,j=1 \\ i\neq j}}^k \int_{\partial \mathbb R_+^{N} } U_{x_i,\Lambda}^{2^\#-1}U_{x_j,\Lambda} = k \sum_{j=2}^k \int_{\mathbb R^{N-1} }U_{x_1,\Lambda}^{2^\#-1}U_{x_j,\Lambda}\nonumber\\
=&k \left[\sum_{j=2}^k \F{C_{3N}}{\Lambda^{N-2}|\bar x_1-\bar
x_j|^{N-2}}+O\bigg(\sum_{j=2}^k\F{\ln\Lambda|\bar x_i-\bar
x_1|}{\Lambda^{N-1}|\bar x_1-\bar x_j|^{N-1}}\bigg)\right],
\end{align}
where $C_{3N}=(N-2)^{N-1}\int_{\mathbb R^{N-1}}\F{\mathrm dz}{(1+|z|^2)^\F{N}{2}}$.
In fact, denote that $d_j=|\bar x_1-\bar x_j|$, then Taylor's expansion tells us that, in $B_\F{d_j}{2}(\bar x_1)\subset\partial\mathbb R_+^N= \mathbb R^{N-1}$ and for large $d_j$,
\begin{align}\label{a.10}
&\left(\F{1}{1+\Lambda^2|\bar y - \bar
x_j|^2}\right)^\F{N-2}{2}\nonumber\\
 =& \left(\F{1}{1+\Lambda^2|\bar
x_1 - \bar x_j|^2}\right)^\F{N-2}{2}
+O\left(\F{|\bar y-\bar x_1|}{\Lambda^{N-2}|\bar x_1-\bar x_j|^{N-1}}\right)\nonumber\\
=&\F{1}{\Lambda^{N-2}|\bar x_1 - \bar x_j|^{N-2}} + O\left(\F{|\bar
y-\bar x_1|}{\Lambda^{N-2}|\bar x_1-\bar x_j|^{N-1}}\right)
+O\left(\F{1}{\Lambda^N|\bar x_1 - \bar x_j|^N}\right).
\end{align}
Thus
\begin{align*}
\int_{B_\F{d_j}{2}(\bar x_1)}U_{x_1,\Lambda}^{2^\#-1}U_{x_j,\Lambda}=& \F{(N-2)^{N-1}}{\Lambda^{N-2}|\bar x_1-\bar x_j|^{N-2}}\int_{\mathbb R^{N-1}}\F{\mathrm dz}{(1+|z|^2)^\F{N}{2}}\\
&+O\left(\F{\ln\Lambda|\bar x_1-\bar x_j|}{\Lambda^{N-1}|\bar
x_1-\bar x_j|^{N-1}}\right).
\end{align*}
In $B_\F{d_j}{2}(\bar x_j)$, since $|\bar y-\bar x_1|\geq \F{|\bar
x_1-\bar x_j|}{2}$ and $|\bar y-\bar x_1|\geq |\bar y-\bar x_j|$, it
is easy to know that \BEN \left(\F{1}{1+\Lambda^2|\bar y - \bar
x_1|^2}\right)^\F{N}{2} \leq \left(\F{1}{1+\F{\Lambda^2}{4}|\bar x_1
- \bar x_j|^2}\right)^\F{N-1}{2}\left(\F{1}{1+\Lambda^2|\bar y -
\bar x_j|^2}\right)^\F{1}{2}, \EEN therefore we have \BEN
\int_{B_\F{d_j}{2}(\bar
x_j)}U_{x_1,\Lambda}^{2^\#-1}U_{x_j,\Lambda}=O\left(\F{\ln\Lambda|\bar
x_1-\bar x_j|}{\Lambda^{N-1}|\bar x_1-\bar x_j|^{N-1}}\right). \EEN
In $\mathbb R^{N-1}\setminus B_\F{d_j}{2}(\bar x_1)\cup
B_\F{d_j}{2}(\bar x_j)$, it holds that \BEN \int_{\mathbb
R^{N-1}\setminus B_\F{d_j}{2}(\bar x_1)\cup B_\F{d_j}{2}(\bar
x_j)}U_{x_1,\Lambda}^{2^\#-1}U_{x_j,\Lambda}
=O\left(\F{1}{\Lambda^{N-1}|\bar x_1-\bar x_j|^{N-1}}\right). \EEN
From (\ref{a.1}) and (\ref{a.2}), we finally obtain that
\begin{align}
&\int_{\mathbb R^N } |D W_{r,\Lambda}|^2= \sum_{j=1}^k \sum_{i=1}^k \int_{\partial\mathbb R^N_+ } U_{x_j,\Lambda}^{2^\#-1}U_{x_i,\Lambda} \nonumber\\
=&\ k \Bigl(\int_{\partial\mathbb R^N_+} U^{2^\#}_{0,1} +\sum_{j=2}^k \int_{\partial\mathbb R^N_+ } U_{x_1,\Lambda}^{2^\#-1}U_{x_j,\Lambda}\Bigr) \nonumber\\
=&\ k \left[A_{N} +\sum_{j=2}^k\frac{ C_{3N}}{\Lambda^{N-2}|\bar x_1- \bar x_j|^{N-2}}
+O\bigg(\sum_{j=2}^k\F{\ln\Lambda|\bar x_1-\bar x_j|}{\Lambda^{N-1}|\bar x_1-\bar x_j|^{N-1}}\bigg)\right]. \label{a.8}
\end{align}\par

Let
\[
\Omega_{j}=\left\{ \bar y:\; \bar y=(\bar y',\bar y'')\in \mathbb
R^2\times \mathbb R^{N-3}=\partial\mathbb R_+^N,\;
  \left\langle \frac {\bar y'}{|\bar y'|}, \frac{x_j}{|x_j|}\right\rangle \ge \cos\frac {\pi}k\right\}.
\]
Then, from Taylor's expansion we obtain that
\begin{align}
&\ \int_{\partial\mathbb R_+^N } K\bigl(\frac {|\bar y|}\mu\bigr) |W_{r,\Lambda}|^{2^\#}= k\int_{\Omega_1}K\bigl(\frac {|\bar y|}\mu\bigr) |W_{r,\Lambda}|^{2^\#} \nonumber\\
=&\ k\Bigg[\int_{\Omega_1}K\bigl(\frac {|\bar y|}\mu\bigr) U_{x_1,\Lambda}^{2^\#}
+2^\#\int_{\Omega_1}K\bigl(\frac {|y|}\mu\bigr)
\sum_{i=2}^k U_{x_1,\Lambda}^{2^\#-1}U_{x_i,\Lambda} \nonumber\\
& +O\left(\int_{\Omega_1} U_{x_1,\Lambda}^{2^\#-2}\Big( \sum_{i=2}^k
U_{x_i,\Lambda}\Big)^{2}\right)+O\left(\int_{\Omega_1}\Big(\sum_{i=2}^k
U_{x_i,\Lambda}\Big)^{2^\#}\right)\Bigg] \label{a.3}.
\end{align}
First, let us estimate the remainders. Note that for $\bar y\in
\Omega_1$, it holds that $|\bar y-\bar x_i|\geq |\bar y-\bar x_1|$
and $|\bar y-\bar x_i|\geq \F{1}{2}|\bar x_i-\bar x_1|$. Thus we
know, for any $0<\alpha<N-2$, that \BEN \sum_{i=2}^k
U_{x_i,\Lambda}\leq  \frac{C}{(1+|\bar y-\bar
x_1|)^{N-2-\alpha}}\sum_{i=2}^k \frac1{|\bar x_i-\bar x_1|^\alpha},
\EEN
 and it is not difficult to check, for any $\alpha >1$, that \BEN
\sum_{j=1}^k \F{1}{|\bar x_1-\bar x_j|^\alpha} = \sum_{j=1}^k
\F{1}{r^\alpha \sin^\alpha\F{(j-1)\pi}{k}} =
O\left(\Big(\F{k}{\mu}\Big)^\alpha\right)=O(\F{1}{\mu^\F{m\alpha}{N-2}}) \EEN If we select the
constant $\alpha$ with $\frac{(N-2)}{2}<\alpha=\F{m+\sigma}{m}\cdot\F{N-2}{2} <\F{N-1}{2}$ ($N\geq
5$), then
\begin{align}\label{a.4}
&\int_{\Omega_1} U_{x_1,\Lambda}^{2^\#-2}\Big(\sum_{i=2}^k U_{x_i,\Lambda}\Big)^{2}\nonumber\\
\leq& C \left(\F{k}{\mu}\right)^{2\alpha} \int_{\Omega_1}
\frac{1}{(1+|\bar y-\bar x_1|)^{2+2(N-2-\alpha)}} =O\left(
\frac 1 {\mu^{m+\sigma}}\right).
\end{align}
In addition, we may also choose $\alpha$ independently such that $\F{(N-2)^2}{2(N-1)}<\alpha=\F{m+\sigma}{m}\cdot\F{(N-2)^2}{2(N-1)}<\F{N-2}{2}$ ($N\geq 5$)
and then acquire that
\BEN
\int_{\Omega_1}\left(\sum_{i=2}^k U_{x_i,\Lambda}\right)^{2^\#} =O\left(\frac 1{\mu^{m+\sigma}}\right).
\EEN
Next we will calculate the second term in (\ref{a.3}). It is easy to show as in (\ref{a.2}) that
\begin{align}
&\int_{\Omega_1}K\bigl(\frac {|\bar y|}\mu\bigr)
\sum_{i=2}^k U_{x_1,\Lambda}^{2^\#-1}U_{x_i,\Lambda}   \nonumber\\
=& \int_{\Omega_1}
\sum_{i=2}^k U_{x_1,\Lambda}^{2^\#-1}U_{x_i,\Lambda}
+\int_{\Omega_1}\left(K\bigl(\frac {|\bar y|}\mu\bigr)-1\right)
\sum_{i=2}^k U_{x_1,\Lambda}^{2^\#-1}U_{x_i,\Lambda}   \nonumber\\
=&\sum_{i=2}^k\frac{ C_{3N} }{\Lambda^{N-2}|x_1-x_i|^{N-2}}+O\left(\frac 1{\mu^{m+\sigma}}\right). \label{a.5}
\end{align}
Finally the first term in (\ref{a.3})
\begin{align}
&\int_{\Omega_1}K\bigl(\frac {|\bar y|}\mu\bigr) U_{x_1,\Lambda}^{2^\#} = \int_{\{||\bar y|-\mu r_0|\leq \mu\delta\}\cap\Omega_1}K\bigl(\frac {|\bar y|}\mu\bigr) U_{x_1,\Lambda}^{2^\#} + O\left(\F{k^{N-1}}{\Lambda^{2N-2}\mu^{N-1}}\right)  \nonumber\\
\nonumber\\
=&\int_{\{|\bar y-\mu r_0|\leq \mu\delta\}\cap\Omega_1}  U_{x_1,\Lambda}^{2^\#}-\frac{c_0 }{\mu^m}\int_{\{||\bar y|-\mu r_0|\leq \mu\delta\}\cap\Omega_1}||\bar y|-\mu r_0|^mU_{x_1,\Lambda}^{2^\#}  \nonumber\\
\nonumber\\ & +O\left(\mu^{-m-\theta}\int_{\{||\bar y|-\mu r_0|\leq
\mu\delta\}\cap\Omega_1}||\bar y|-\mu
r_0|^{m+\theta}U_{x_1,\Lambda}^{2^\#}\right)
+ O\left(\F{k^{N-1}}{\Lambda^{2N-2}\mu^{N-1}}\right)  \nonumber\\
\nonumber\\
=&A_N-\frac{c_0 }{\mu^m}\int_{\partial\mathbb R_+^{N}} ||\bar y|-\mu
r_0|^m U_{x_1,\Lambda}^{2^\#} \mathrm d\bar y
+O\left(\frac1{\mu^{m+\theta}}\right)+ O\left(\F{k^{N-1}}{\Lambda^{2N-2}\mu^{N-1}}\right)  \nonumber\\
\nonumber\\
=&A_N -\frac{c_0 }{\mu^m}\int_{\partial\mathbb R_+^{N}} ||\bar
y-\bar x_1|-\mu r_0|^m U_{0,\Lambda}^{2^\#} \mathrm d\bar y
+O\left(\frac1{\mu^{m+\theta}}+ \F{k^{N-1}}{\mu^{N-1}}\right).
\label{a.6}
\end{align}
But
\newpage
\begin{align*}
 & \F{1}{\mu^m}\int_{\partial\mathbb R_+^{N}\setminus B_{\F{|\bar x_1|}{2}}(0)}||\bar y-\bar x_1|-\mu r_0|^m U_{0,\Lambda}^{2^\#}\mathrm d\bar y\\
  \leq &C \int_{\partial\mathbb R_+^{N}\setminus B_{\F{|\bar x_1|}{2}}(0)}\left(\F{|\bar y|^m}{\mu^m}+1\right)\F{\mathrm d\bar y}{(1+\Lambda^2|\bar
  y|^2)^{N-1}}\\
    \leq & \F{C}{\mu^{N-1}}.
\end{align*}
On the other hand, if $\bar y\in B_{\F{|\bar x_1|}{2}}(0)$, $\bar
y=(\bar y_1,\bar y^*)$, $\bar y^*=(\bar y_2,\cdots,\bar y_{N-1})$,
then $|\bar x_1|-\bar y_1 \geq \F{|\bar x_1|}{2}>0$. So, as $|\bar
x_1|$ becomes large, \BEN |\bar y-\bar x_1|=|\bar x_1|-\bar y_1 +
O\left(\F{|\bar y^*|^2}{|\bar x_1|-\bar y_1}\right)=|\bar x_1|-\bar
y_1+O\left(\F{|\bar y^*|^2}{|\bar x_1|}\right). \EEN As a result,
Taylor's expansion says, for $m\geq2$,  that
\begin{align*}
& ||\bar y-\bar x_1|-\mu r_0|^m= \left| |\bar x_1|-\bar y_1 +O(\F{|\bar y^*|^2}{|\bar x_1|})-\mu r_0\right|^m\\
=&\ |\bar y_1|^m + m|\bar y_1|^{m-2} \bar y_1 \left[\mu r_0 -|\bar x_1|+O(\F{|\bar y^*|^2}{|\bar x_1|})\right] \\
& + \frac12 m(m-1) |\bar y_1|^{m-2}\left[\mu r_0 -|\bar x_1|+O(\F{|\bar y^*|^2}{|\bar x_1|})\right]^2 \\
&  + O\left( |y_1|^{m-2-\tilde\sigma}\left|\mu r_0 -|\bar x_1|+O(\F{|\bar y^*|^2}{|\bar x_1|})\right|^{2+\tilde\sigma}\right) \qquad (\tilde\sigma=\min\{m-2,1\})\\
&  + O\left( \left|\mu r_0 -|\bar x_1|+O(\F{|\bar y^*|^2}{|\bar
x_1|})\right|^{m}\right).
\end{align*}
Thus, using \BEN \int_{B_\F{|\bar x_1|}{2}(0)}\F{|\bar y_1|^{m-2}
\bar y_1}{(1+\Lambda^2|\bar y|^2)^{N-1}}\mathrm d\bar y =0, \EEN we
obtain that, since $m<N-2$,
\begin{align}
& \F{1}{\mu^m}\int_{\partial \mathbb R_+^{N}}||y-x_1|-\mu r_0|^mU_{0,\Lambda}^{2^\#}\nonumber\\
= &\F{1}{\mu^m}\int_{B_\F{|\bar x_1|}{2}}||y-x_1|-\mu r_0|^mU_{0,\Lambda}^{2^\#} +O\left(\F{1}{\mu^{N-1}}\right)   \nonumber\\
=& \F{1}{\mu^m}\int_{\partial \mathbb R_+^{N}}|\bar y_1|^{m}U_{0,\Lambda}^{2^\#}\mathrm d\bar y
+\frac{ m(m-1)}{2\mu^m} \int_{\partial \mathbb R_+^{N}}|\bar y_1|^{m-2}(\mu r_0 -|x_1|)^2U_{0,\Lambda}^{2^\#} \mathrm d\bar y \nonumber\\
& +O\left(\F{1}{\mu^{m}}|\mu r_0-r|^{2+\tilde\sigma}+ \F{1}{\mu^{N-1}}\right)\nonumber \\
=&\ \F{C_{1N}}{\Lambda^m \mu^m}
+\frac{C_{2N}}{\Lambda^{m-2}\mu^m}(\mu r_0 -|x_1|)^2 +O\left(\F{1}{\mu^{m}}|\mu r_0-r|^{2+\tilde\sigma}+ \F{1}{\mu^{N-1}}\right), \label{a.7}
\end{align}
where
\[C_{1N}=(N-2)^{N-1}\int_{\mathbb R^{N-1}}\F{|\bar
y_1|^{m}\mathrm d\bar y}{(1+|\bar y|^2)^{N-1}}
\]
 and
 \[C_{2N}=\frac{
m(m-1)(N-2)^{N-1}}{2} \int_{\mathbb R^{N-1}}\F{|\bar
y_1|^{m-2}}{(1+|\bar y|^2)^{N-1}}\mathrm d\bar y. \]
 Thus, from
(\ref{a.3})--(\ref{a.7}) we have proved
\begin{align}
& \int_{\mathbb R^{N-1} } K\bigl(\frac {|\bar y|}\mu\bigr) |W_{r,\Lambda}|^{2^\#} \nonumber\\
=&\ k\Bigg[A_N -\F{C_{1N}}{\Lambda^m \mu^m} -\frac{C_{2N}}{\Lambda^{m-2}\mu^m}(\mu r_0 -|x_1|)^2 \nonumber\\
&+2^\#\sum_{i=2}^k\frac{ C_{3N} }{\Lambda^{N-2}|x_1-x_j|^{N-2}}+
O\Bigl(\F{1}{\mu^{m}}|\mu r_0-r|^{2+\tilde\sigma}+(\F{1}{\mu})^{m+\sigma}\Bigr)\Bigg].
\label{a.9}
\end{align}\par
The proposition is concluded from (\ref{a.8}) and (\ref{a.9}) by
setting $A=(\F{1}{2}-\F{1}{2^\#})A_N$, $B_1=\F{C_{1N}}{2^\#}$,
$B_2=\F{C_{2N}}{2^\#}$ and $B_3=\F{C_{3N}}{2}$.
\end{proof}

\begin{proposition}\label{pa3}
We have
\begin{align*}
\frac{\partial I(W_{r,\Lambda})}{\partial \Lambda}=& k\Bigl[
-\frac{m B_1}{\Lambda^{m+1}\mu^m}
+\sum_{i=2}^k\frac{ B_3 (N-2) }{\Lambda^{N-1}|x_1-x_j|^{N-2}}\\
&+O\bigl(\frac1{\mu^{m+\sigma}}+\frac{1}{\mu^m} |\mu r_0
-r|^2\bigr) \Bigr],
\end{align*}
where $B_i$ ($i=1,3$) are the same positive constants as in
Proposition \ref{pa2}.
\end{proposition}

\begin{proof}
The proof of this proposition is similar to that of Proposition~\ref{pa2}. So we just sketch it.\par

It is not difficult to get \BE\label{a.12} \frac{\partial
I(W_{r,\Lambda})}{\partial\Lambda}= k\Biggl[\frac12 \sum_{i=2}^k
\F{\partial}{\partial\Lambda} \int_{\partial\mathbb
R_+^{N}}U_{x_1,\Lambda}^{2^\#-1} U_{x_i,\Lambda}
-\frac1{2^\#}\F{\partial}{\partial\Lambda} \int_{\Omega_1}
K\bigl(\frac {|\bar y|}\mu\bigr) W_{r,\Lambda}^{2^\#} \Biggr]. \EE
Note that
\[
\left.\F{\partial
U_{x_j,\Lambda}}{\partial\Lambda}\right|_{\partial\mathbb R_+^N}
=\F{(N-2)}{2\Lambda}\F{1-\Lambda^2|\bar y-\bar
x_j|^2}{1+\Lambda^2|\bar y-\bar x_j|^2}\left.
U_{x_j,\Lambda}\right|_{\partial\mathbb R_+^N},
\]
hence
\begin{align}\label{a.11}
\F{\partial}{\partial\Lambda}\int_{\partial\mathbb R^N_+} U_{x_1,\Lambda}^{2^\#-1}U_{x_{i,\Lambda}} = &\F{\partial}{\partial\Lambda}\int_{\partial\mathbb R^N_+} U_{x_1,\Lambda}^{2^\#-1}\F{(N-2)^\F{N-2}{2}}{\Lambda^\F{N-2}{2}|\bar x_i-\bar x_1|^{N-2}} \nonumber\\
&+\F{\partial}{\partial\Lambda}\int_{\partial\mathbb R^N_+}
U_{x_1,\Lambda}^{2^\#-1}
\left(U_{x_{i,\Lambda}}-\F{(N-2)^\F{N-2}{2}}{\Lambda^\F{N-2}{2}|\bar
x_i-\bar x_1|^{N-2}}\right).
\end{align}

In $B_\F{d_i}{2}(\bar x_1)$, recalling (\ref{a.10}) and using
\begin{align*}
&\F{\partial}{\partial\Lambda}\left(U_{x_{i,\Lambda}}-\F{(N-2)^\F{N-2}{2}}{\Lambda^\F{N-2}{2}|\bar x_i-\bar x_1|^{N-2}}\right)\\
=&O\left(\F{|\bar y-\bar x_1|}{\Lambda^\F{N}{2}|\bar x_i-\bar
x_1|^{N-1}}\right)+O\left(\F{1}{\Lambda^\F{N+4}{2}|\bar x_i-\bar
x_1|^{N}}\right),
\end{align*}
we have that
\BEN
\F{\partial}{\partial\Lambda}\int_{B_\F{d_i}{2}(\bar x_1)} U_{x_1,\Lambda}^{2^\#-1} \left(U_{x_{i,\Lambda}}-\F{(N-2)^\F{N-2}{2}}{\Lambda^\F{N-2}{2}|\bar x_i-\bar x_1|^{N-2}}\right)
=O\left(\F{\ln\Lambda|\bar x_i-\bar x_1|}{\Lambda^{N}|\bar x_i-\bar x_1|^{N-1}}\right).
\EEN
Similar as the proof of Proposition \ref{pa2}, it is also easy to check that
\begin{align*}
&\F{\partial}{\partial\Lambda}\int_{B_\F{d_i}{2}(\bar x_i)} U_{x_1,\Lambda}^{2^\#-1} \left(U_{x_{i,\Lambda}}-\F{(N-2)^\F{N-2}{2}}{\Lambda^\F{N-2}{2}|\bar x_i-\bar x_1|^{N-2}}\right)
=O\left(\F{\ln\Lambda|\bar x_i-\bar x_1|}{\Lambda^{N}|\bar x_i-\bar x_1|^{N-1}}\right),\\
&\F{\partial}{\partial\Lambda}\int_{\partial\mathbb R^N_+\setminus B_\F{d_i}{2}(\bar x_i)\cup B_\F{d_i}{2}(\bar x_1)} U_{x_1,\Lambda}^{2^\#-1} \left(U_{x_{i,\Lambda}}-\F{(N-2)^\F{N-2}{2}}{\Lambda^\F{N-2}{2}|\bar x_i-\bar x_1|^{N-2}}\right)\\
 = &O\left(\F{1}{\Lambda^{N}|\bar x_i-\bar x_1|^{N-1}}\right).
\end{align*}
Thus from (\ref{a.11}) we get that
\begin{align*}
&\F{\partial}{\partial\Lambda}\int_{\partial\mathbb R^N_+} U_{x_1,\Lambda}^{2^\#-1}U_{x_{i,\Lambda}} \\
=& \F{\partial}{\partial\Lambda}\int_{\partial\mathbb R^N_+}
U_{x_1,\Lambda}^{2^\#-1}\F{(N-2)^\F{N-2}{2}}{\Lambda^\F{N-2}{2}|\bar
x_i-\bar x_1|^{N-2}} +O\left(\F{\ln\Lambda|\bar x_i-\bar x_1|}{\Lambda^{N}|\bar x_i-\bar x_1|^{N-1}}\right) \\
= &-\F{(N-2)C_{3N}}{\Lambda^{N-1}|\bar x_i-\bar
x_1|^{N-2}}+O\left(\F{\ln\Lambda|\bar x_i-\bar
x_1|}{\Lambda^{N}|\bar x_i-\bar x_1|^{N-1}}\right).
\end{align*}

As for the terms in the right side of (\ref{a.12}), direct computations show that
\begin{align}
&~\F{\partial}{\partial\Lambda}\int_{\Omega_1}K\bigl(\frac {|\bar y|}\mu\bigr)
 U_{x_1,\Lambda}^{2^\#-1}U_{x_i,\Lambda}   \nonumber\\
=&~ \F{\partial}{\partial\Lambda}\int_{\Omega_1}
 U_{x_1,\Lambda}^{2^\#-1}U_{x_i,\Lambda}
+\F{\partial}{\partial\Lambda}\int_{\Omega_1}\left(K\bigl(\frac {|\bar y|}\mu\bigr)-1\right)
 U_{x_1,\Lambda}^{2^\#-1}U_{x_i,\Lambda}   \nonumber\\
=&~\F{\partial}{\partial\Lambda}\left(\int_{\partial\mathbb R^N_+}-\int_{\partial\mathbb R^N_+\setminus\Omega_1\cup B_\F{d_i}{2}(\bar x_i)}-\int_{ B_\F{d_i}{2}(\bar x_i)}\right)
 U_{x_1,\Lambda}^{2^\#-1}U_{x_i,\Lambda} \nonumber\\
& +\F{\partial}{\partial\Lambda}\int_{\Omega_1}\left(K\bigl(\frac
{|\bar y|}\mu\bigr)-1\right)
 U_{x_1,\Lambda}^{2^\#-1}U_{x_i,\Lambda} \nonumber\\
=&~-\F{(N-2)C_{3N}}{\Lambda^{N-1}|\bar x_i-\bar x_1|^{N-2}}+O\left(
\bigl(\frac k \mu\bigr)^{N-2+\sigma}\right). \label{a.15}
\end{align}
The last equality is due to that, because of the condition on the
function $K$,
\begin{align*}
&\ \F{\partial}{\partial\Lambda}\int_{\Omega_1}\left(K\bigl(\frac {|\bar y|}\mu\bigr)-1\right)
 U_{x_1,\Lambda}^{2^\#-1}U_{x_i,\Lambda}\\
 =&\ \F{\partial}{\partial\Lambda}\left(\int_{\Omega_1\cap \left\{\left||\bar y|-\mu r_0\right| \leq \mu^{1-\sigma}\right\}}+\int_{\Omega_1\cap \left\{\left||\bar y|-\mu r_0\right| \geq \mu^{1-\sigma}\right\}}\right)U_{x_1,\Lambda}^{2^\#-1}U_{x_i,\Lambda}\\
 = &\ O\left(\F{1}{\mu^{m\sigma}\Lambda^{N-1}|\bar x_i-\bar x_1|^{N-2}}\right)+O\left(\F{\ln\Lambda|\bar x_i-\bar x_1|}{\Lambda^{N}|\bar x_i-\bar x_1|^{N-1}}\right).
\end{align*}
By the similar estimates as in getting (\ref{a.6}) and (\ref{a.7}), we have that
\begin{align*}
&\F{\partial}{\partial\Lambda}\int_{\Omega_1}K\bigl(\frac {|\bar y|}\mu\bigr) U_{x_1,\Lambda}^{2^\#} = \F{\partial}{\partial\Lambda}\int_{\{||\bar y|-\mu r_0|\leq \mu\delta\}\cap\Omega_1}K\bigl(\frac {|\bar y|}\mu\bigr) U_{x_1,\Lambda}^{2^\#} + O\left(\F{k^{N-1}}{\mu^{N-1}}\right)  \nonumber\\
=&\F{\partial}{\partial\Lambda}\int_{\{|\bar y-\mu r_0|\leq \mu\delta\}\cap\Omega_1}  U_{x_1,\Lambda}^{2^\#}-\frac{c_0 }{\mu^m}\F{\partial}{\partial\Lambda}\int_{\{||\bar y|-\mu r_0|\leq \mu\delta\}\cap\Omega_1}||y|-\mu r_0|^mU_{x_1,\Lambda}^{2^\#}  \nonumber\\
&+O\left(\frac1{\mu^{m+\theta}}+ \F{k^{N-1}}{\mu^{N-1}}\right)  \nonumber\\
=&-\frac{c_0 }{\mu^m}\F{\partial}{\partial\Lambda}\int_{\mathbb R^{N-1}} ||\bar y|-\mu r_0|^m U_{x_1,\Lambda}^{2^\#} \mathrm d\bar y
+O\left(\frac1{\mu^{m+\sigma}}+ \F{k^{N-1}}{\mu^{N-1}}\right) \nonumber\\
=& -\frac{c_0 }{\mu^m}\F{\partial}{\partial\Lambda}\int_{\mathbb R^{N-1}}||\bar y-\bar x_1|-\mu r_0|^m U_{0,\Lambda}^{2^\#}\mathrm d\bar y
+O\left(\frac1{\mu^{m+\sigma}}+ \F{k^{N-1}}{\mu^{N-1}}\right). \nonumber\\
=&\ \frac{ mC_{1N} }{\Lambda^{m+1}\mu^m}
+O\left(\frac1{\mu^{m+\sigma}}+\F{1}{\mu^{m}}|\mu r_0-r|^{2}\right).
\end{align*}
The remaining estimates of this proposition are similar to the previous one. We omit the details.
\end{proof}

\subsection{Basic Estimates}

For each fixed $i$ and $j$, $i\ne j$, consider the following
function

\begin{equation}\label{aa1}
g_{ij}(y)= \frac{1}{(1+|y-x_j|)^{\alpha}}\frac{1}{(1+|y-x_i|)^{\beta}},
\end{equation}
where $\alpha>0$ and $\beta>0$ are two constants.

Then we have the following lemma whose proof can be found in
 Appendix~B in \cite{WY}.

\begin{lemma}\label{laa1}
For any constant $0\leq\sigma\le \min(\alpha,\beta)$, there is a
constant $C>0$, such that
\[
g_{ij}(y)\le \frac{C}{|x_i-x_j|^\sigma}\left[\frac{1}{(1+|y-x_i|)^{\alpha+\beta-\sigma}}+
\frac{1}{(1+|y-x_j|)^{\alpha+\beta-\sigma}}\right].
\]
\end{lemma}

\begin{lemma}\label{laa2}
For any constant $0<\sigma<N-2$, there is a constant $C>0$, such
that for any $y \in \overline{\mathbb R_+^{N}}$,

\[
\int_{\partial\mathbb R_+^N}
\frac1{|y-z|^{N-2}}\frac1{(1+|z|)^{1+\sigma}}\,\mathrm d \bar z\le
\frac C{(1+|y|)^{\sigma}},
\]
where $z=(\bar z, 0)=\mathbb R^{N-1} \times \{0\}\in \partial\mathbb
R_+^N.$

\end{lemma}

The result is well known. Readers may refer to Appendix~B in
\cite{WY} to find almost the same proof.

\begin{lemma}\label{laa3}
Suppose that $N\ge 5$. Then  for any $y \in \overline{ \mathbb
R_+^{N}}$, we have that
\begin{align*}
 &\int_{\partial\mathbb R_+^N} \F{1}{|y-z|^{N-2}}W_{r,\Lambda}^{2^\#-2}(z) \sum_{j=1}^k\frac1{(1+|z-x_j|)^{\F{N-2}{2}-\F{m}{N-2}+\tau}}\mathrm d \bar z \\
\leq &\sum_{i=1}^k \F{C}{(1+|y-x_i|)^{\F{N-2}{2}-\F{m}{N-2}+\tau}}.
\end{align*}
\end{lemma}

\begin{proof}
Note that for any $\beta\geq \F{N-2-m}{N-2}$ and fixed $\ell$, as $k\to\infty$
\begin{align*}
&\sum_{i\ne \ell}\frac1{|x_i-x_\ell|^\beta}=\F{1}{2^\beta}\sum_{i\neq \ell} \F{1}{r^\beta \sin^\beta\F{|i-\ell|\pi}{k}} \\
\le &\frac{C k^\beta}{\mu^\beta}\sum_{i=1}^k \frac1{i^\beta}\leq
\begin{cases}
\F{Ck^\beta}{\mu^\beta}=O(\mu^{-\F{m\beta}{N-2}}) \qquad & \beta>1, \\
\F{Ck^\beta\ln k}{\mu^\beta}=O(\mu^{-\F{m\beta}{N-2}}\ln\mu) \quad & \beta=1, \\
\F{C k}{\mu^\beta}=O(\mu^{-(\beta-\F{N-2-m}{N-2})}) \quad & \beta<1.
\end{cases}
\end{align*}
In $\Omega_\ell$, we have $|z-x_j|=|\bar z-\bar x_j|\geq | z-
x_\ell|$ and $| z- x_j|\geq | x_j- x_\ell|$ for any $j\neq \ell$.
Thus for any $\F{N-2-m}{N-2}\leq\alpha\leq N-2$, it holds \BEN
\sum_{j\neq \ell}\F{1}{(1+| z- x_j|)^{N-2}} \leq \F{1}{(1+| z-
x_\ell|)^{N-2-\alpha}} \sum_{j\neq \ell}\F{1}{| x_j- x_\ell|^\alpha}
.. \EEN Thus in $\Omega_\ell$ we have
\begin{align*}
&W_{r,\Lambda}^{2^\#-2}(z) \leq \F{C}{(1+| z- x_\ell|)^{2-\F{2\alpha}{N-2}}},  \\
&\sum_{j=1}^k\frac1{(1+| z- x_j|)^{\F{N}{2}-\F{m}{N-2}+\tau}} \leq
\F{C}{(1+| z- x_\ell|)^{\F{N}{2}-\F{m}{N-2}+\tau-\alpha}}.
\end{align*}
As a result, we find for  $z\in\Omega_\ell$ that
\BEN
W_{r,\Lambda}^{2^\#-2}(z) \sum_{j=1}^k\frac1{(1+|z-x_j|)^{\F{N}{2}-\F{m}{N-2}+\tau}}
\leq \F{C}{(1+|z-x_\ell|)^{\F{N+2}{2}-\F{m}{N-2}+1-\F{N\alpha}{N-2}+\tau}}.
\EEN
It gives that, for $\alpha=\F{N-2-m}{N-2}$, since $\partial\mathbb R_+^N=\underset{i=1}{\overset{k}{\cup}}\Omega_i$,
\begin{align*}
 &\int_{\partial\mathbb R_+^N} \F{1}{|y-z|^{N-2}}W_{r,\Lambda}^{2^\#-2}(z) \sum_{j=1}^k\frac1{(1+|z-x_j|)^{\F{N}{2}-\F{m}{N-2}+\tau}} \mathrm d \bar z\\
 \leq& \sum_{i=1}^k \F{C}{(1+|y-x_i|)^{\F{N}{2}-\F{m}{N-2}+\F{(m-2)N+4}{(N-2)^2}+\tau}}\\
  \leq& \sum_{i=1}^k \F{C}{(1+|y-x_i|)^{\F{N}{2}-\F{m}{N-2}+\tau}}.
\end{align*}
\end{proof}

\end{document}